\documentclass[11pt,a4paper]{article}
\usepackage{fullpage}
\usepackage{amssymb,latexsym,amsmath,amsfonts,amsthm}
\usepackage{graphicx}
\usepackage{color}

\usepackage{overpic}
\usepackage{amsmath}
\usepackage{amsthm}
\usepackage{amssymb}
\usepackage{amsfonts}
\usepackage{hyperref}
\usepackage{indentfirst}
\usepackage{enumitem}

\DeclareMathOperator{\pv}{p.v.}
\DeclareMathOperator{\Ai}{Ai}
\DeclareMathOperator{\local}{local}
\DeclareMathOperator{\glob}{global}
\DeclareMathOperator{\curved}{curved}
\DeclareMathOperator{\vertical}{vertical}
\DeclareMathOperator{\lef}{left}
\DeclareMathOperator{\rig}{right}
\DeclareMathOperator{\out}{outer}
\DeclareMathOperator{\inner}{inner}

\newcommand{\compC}{\mathbb{C}}
\newcommand{\intZ}{\mathbb{Z}}
\newcommand{\realR}{\mathbb{R}}

\renewcommand{\Re}{\mathrm{Re}\,}
\renewcommand{\Im}{\mathrm{Im}\,}
\newcommand{\ds}{\displaystyle}

\newcommand{\ud}{\,\mathrm{d}}
\newcommand{\ie}{i.e.}

\newtheorem{thm}{Theorem}[section]

\newtheorem{lem}[thm]{Lemma}

\theoremstyle{remark}
\newtheorem{rmk}{Remark}
\numberwithin{equation}{section}

\newcommand{\ai}{\textrm{Ai}}

\newcommand{\bigO}{\mathcal{O}}

\newcommand{\eq}{\begin{equation}}
\newcommand{\nq}{\end{equation}}
\newcommand{\eqa}{\begin{eqnarray}}
\newcommand{\nqa}{\end{eqnarray}}

\begin{document}

\title{Bulk and soft-edge universality for singular values of products of Ginibre random matrices}
\author{Dang-Zheng Liu\footnotemark[1] ~ Dong Wang\footnotemark[2] ~ and Lun Zhang\footnotemark[3]}
\maketitle
\renewcommand{\thefootnote}{\fnsymbol{footnote}}
\footnotetext[1]{Key Laboratory of Wu Wen-Tsun Mathematics, Chinese Academy of Sciences, School of Mathematical Sciences, University of Science and Technology of China, Hefei 230026, P. R. China. E-mail: dzliu\symbol{'100}ustc.edu.cn}
\footnotetext[2]{Department of Mathematics, National University of Singapore, 119076, Singapore. E-mail: matwd@nus.edu.sg}
\footnotetext[3] {School of Mathematical Sciences and Shanghai Key Laboratory for Contemporary Applied Mathematics, Fudan University, Shanghai 200433, P. R. China. E-mail:
lunzhang\symbol{'100}fudan.edu.cn}

\begin{abstract}
It has been shown by Akemann, Ipsen and Kieburg that the squared singular values of products of $M$ rectangular random matrices with independent complex Gaussian entries are distributed according to a determinantal point process with a correlation kernel that admits a representation in terms of  Meijer G-functions. We prove the universality of the local statistics of the squared singular values, namely, the bulk universality given by the sine kernel and the edge universality given by the Airy kernel. The proof is based on the asymptotic analysis for the double contour integral representation of the correlation kernel. Our strategy can be generalized to deal with other models of products of random matrices introduced recently and to establish similar universal results. Two more examples are investigated, one is the product of $M$ Ginibre matrices and the inverse of $K$ Ginibre matrices studied by Forrester, and the other one is the product of $M-1$ Ginibre matrices with one truncated unitary matrix considered by Kuijlaars and Stivigny.
\end{abstract}


\section{Introduction and statement of the main results} \label{sec:Intro}

\subsection{Products of Ginibre matrices}

Significant progresses have been achieved recently in the study of products of random matrices, which have important applications in Schr\"{o}dinger operator theory \cite{Bougerol-Lacroix85}, in statistical physics relating to disordered and chaotic dynamical systems \cite{Crisanti-Paladin-Vulpiani93} and in wireless communication like MIMO (multiple-input and multiple-output) networks \cite{Tulino-Verdu04}. Although the pioneering work of Furstenberg and Kesten \cite{Furstenberg-Kesten60} focused on the statistical behavior of individual entries in the product as the number of factors tends to infinity, the recent interest of study lies in the distribution of eigenvalues and singular values of the product of a fixed number of matrices, where the sizes of the matrices tend to infinity. Various methods have been applied to perform the spectral analysis in different regimes. Particularly, the tools from free probability allow one to find the limiting mean eigenvalue distributions as in \cite{Alexeev-Gotze-Tikhomirov10, Banica-Belinschi-Capitaine-Collins11, Biane98, Burda-Janik-Waclaw10, Burda-Jarosz-Livan-Nowak-Swiech10, Furstenberg-Kesten60, Gotze-Kosters-Tikhomirov14,Gotze-Tikhomirov10, Nica-Speicher06, ORourke-Soshnikov11, Penson-Zyczkowski11}. It turns out that, as in the theory of matrix model for a single random matrix, the various limits exhibit a rich and interesting mathematical structure. Most of the results in literature on the model of products of matrices are about the global spectral properties, but local universality is also suggested (cf. \cite{Akemann-Kieburg-Wei13, Forrester14, Forrester-Liu14}). Our work is motivated by the previous results and proves the local universality of the squared singular values.

In this paper, we consider $M\geq 1$ independent complex random matrices $X_j$, $j=1,\ldots,M$, each has size $N_{j} \times N_{j-1}$ with independent and identically distributed standard complex Gaussian entries. These matrices are also known as complex Ginibre random matrices. We then form the product
\begin{equation} \label{Ym}
Y_M = X_M X_{M-1} \cdots X_1.
\end{equation}
For convenience, we assume that $N_1, \dotsc, N_M$ are associated to a large integer parameter $n$, which we interchangeably denote by $N_0$, such that
\begin{equation}
  \min \{N_0, \ldots, N_M \} = N_0 = n,
\end{equation}
and set
\begin{equation} \label{nuj}
  \nu_j = N_j - N_0, \qquad j=0, \ldots, M.
\end{equation}
Clearly, $\nu_0=0$ and $\nu_j\geq 0$ for $j=1,\ldots,M$.

When $M=1$, $Y_1 = X_1$ defines the complex Wishart random matrix and plays a fundamental role in random matrix theory; cf. \cite{Forrester10}. It is well known that the eigenvalues and squared singular values of $Y_1$ form determinantal point processes \cite{Johansson06,Soshnikov00}, and their distributions are expressed in terms of the correlation kernels.  Recent studies find the determinantal structures for the model with general $M$; see \cite{Akemann-Burda12} for the eigenvalues and \cite{Akemann-Ipsen-Kieburg13,Akemann-Kieburg-Wei13} for the squared singular values. Moreover, further investigations reveal that similar determinantal structures also appear in many other models of products of random matrices, such as the products involving inverses of complex Ginibre matrices \cite{Adhikari-Reddy-Reddy-Saha13,Forrester14,Krishnapur06} and products involving truncated unitary matrices \cite{Akemann-Burda-Kieburg-Nagao14,Ipsen-Kieburg14,Kuijlaars-Stivigny14}; see also Section \ref{sec:extensions} below, and the recent review paper \cite{Akemann-Ipsen15} and references therein.

We will focus on the squared singular values of $Y_M$. According to \cite{Akemann-Ipsen-Kieburg13}, the joint probability density function of the squared singular values is given by (see \cite[formula (18)]{Akemann-Ipsen-Kieburg13})
\begin{equation} \label{jpdf}
    P(x_1, \ldots, x_n) =  \frac{1}{Z_n}  \prod_{j < k} (x_k-x_j)\,
        \det \left[ w_{k-1}(x_j) \right]_{j,k=1, \ldots, n},
    \end{equation}
where $x_j>0$, the function $w_k$ is a Meijer G-function (see e.g.\ \cite{Beals-Szmigielski13,Luke69})
\begin{equation} \label{wk}
    w_k(x) = \mathop{{G^{{M,0}}_{{0,M}}}\/}\nolimits\!\left({- \atop \nu_M, \nu_{M-1},  \ldots, \nu_2, \nu_1 +
    k} \Big{|} x\right),
    \end{equation}
and the normalization constant (see \cite[formula (21)]{Akemann-Ipsen-Kieburg13}) is
\[ Z_n = n!\prod_{i=1}^{n}\prod_{j=0}^M \Gamma(i+\nu_j).  \]
Note that the Meijer G-function $w_k(x)$ can be written as a Mellin-Barnes integral
\begin{equation} \label{wkasMB}
    w_k(x) =  \frac{1}{2\pi i} \int_{c-i\infty}^{c+i\infty} \Gamma(s+\nu_1 + k) \prod_{j=2}^{M} \Gamma(s+\nu_j)  x^{-s} \ud s,
    \qquad k=0, 1, \ldots,
\end{equation}
with $c > 0$.

\subsection{The correlation kernel and double integral representation}
The determinantal point process \eqref{jpdf} is a biorthogonal ensemble \cite{Borodin99} with correlation kernel
\begin{equation} \label{Kn}
    K_n(x,y) = \sum_{j=0}^{n-1} \sum_{k=0}^{n-1} x^j (M_n^{-1})_{k,j}
    w_k(y),
    \end{equation}
where $M_n$ is the $n \times n$ matrix of moments
\begin{equation} \label{Mmoment}
    M_n = \begin{pmatrix} \ds \int_0^{\infty} x^j w_k(x) \ud x \end{pmatrix}_{j,k=0, \ldots, n-1}.
    \end{equation}
Alternatively, one can write the correlation kernel as
\begin{equation} \label{def:Kn}
    K_n(x,y) = \sum_{k=0}^{n-1} P_k(x) Q_k(y),
    \end{equation}
where for each $k = 0, 1, \ldots,n-1$, $P_k$ is a monic polynomial of degree $k$ and $Q_k$ is a linear combination of $w_0, \ldots, w_k$, uniquely defined by the orthogonality
\begin{equation} \label{PkQkbio}
    \int_0^{\infty} P_j(x) Q_k(x) \ud x = \delta_{j,k}.
    \end{equation}
Thus the functions $P_k$ and $Q_k$ are the so-called biorthogonal functions. The polynomials $P_k$ are also characterized as multiple orthogonal polynomials \cite{Ismail09} with respect to the first $M$ weight functions $w_j$, $j=0,\ldots,M-1$, as shown in  \cite{Kuijlaars-Zhang14}.

It turns out that $P_k$ and $Q_k$ are Meijer G-functions \cite{Akemann-Kieburg-Wei13}, and then they have contour integral representations. Then it is shown in \cite[Proposition 5.1]{Kuijlaars-Zhang14} that the correlation kernel admits the following double contour integral representation
\begin{equation} \label{Knintegral}
    K_n(x,y) =  \frac{1}{(2\pi i)^2} \int_{-1/2-i\infty}^{-1/2+i\infty} \ud s \oint_{\Sigma}  \ud t
        \prod_{j=0}^M   \frac{\Gamma(s+\nu_j+1)}{\Gamma(t+\nu_j+ 1)}
            \frac{\Gamma(t-n+1)}{\Gamma(s-n+1)}
        \frac{x^t y^{-s-1}}{s-t},
\end{equation}
where $\Sigma$ is a closed contour going around $0, 1, \ldots, n-1$ in the positive direction and  $\Re t > -1/2$ for $t \in \Sigma$. The choices of these two contours are not unique. We can, and indeed will, make some deformations in our later analysis.

\subsection{Limiting mean density}

The first step of the study of the correlation kernel is to compute the $1$-point correlation function, which is also known as the mean density of the model. This global result is also the basis of our proof of the local universality \cite{Kuijlaars11}. As mentioned at the very beginning, the limiting mean density/spectral  distribution of the squared singular values for $Y_M$ is well understood using tools from free probability; see also recent work \cite{Neuschel14,Zhang13} for the study from the polynomials $P_k$. It turns out that, after proper scaling, the limiting mean density is recognized as the Fuss-Catalan distribution \cite{Alexeev-Gotze-Tikhomirov10, Banica-Belinschi-Capitaine-Collins11, Nica-Speicher06}, i.e., its $k$-th moment is given by the Fuss-Catalan number
\begin{equation}
\frac{1}{Mk+1}\binom{Mk+k}{k},\quad k=0,1,2,\ldots.
\end{equation}
The (rescaled) limiting mean density is supported on an interval $[0, (M + 1)^{M + 1}/M^M]$, with explicit form given in terms of Meijer G-functions \cite{Penson-Zyczkowski11} or multivariate integrals \cite{Liu-Song-Wang11}. Probably the simplest form of the density function for general $M$ is expressed by the following parametrization of the argument \cite{Biane98,Haagerup-Moller13,Neuschel14}:
\begin{equation}\label{eq:para x}
x=\frac{(\sin ((M+1)\varphi))^{M+1}}{\sin\varphi (\sin( M\varphi))^M}, \quad 0<\varphi<\frac{\pi}{M+1}.
\end{equation}
It is readily seen that this parametrization is a strictly decreasing function of $\varphi$, thus gives a  one-to-one mapping from $(0, \pi/(M+1))$ to $(0, (M + 1)^{M + 1}/M^M)$. The density function in terms of $\varphi$ is then given by
\begin{equation}\label{eq:density}
  \rho(\varphi) = \frac{1}{\pi x} \frac{\sin((M + 1)\varphi)}{\sin(M\varphi)} \sin\varphi = \frac{1}{\pi}\frac{(\sin\varphi)^2(\sin (M\varphi))^{M-1}}{ (\sin ((M+1)\varphi))^M}.
\end{equation}

From \eqref{eq:para x} and \eqref{eq:density}, one can check (cf. \cite{Forrester-Liu14}) that the density blows up with a rate $x^{-M/(M+1)}$ near the origin (hard edge), while vanishes as a square root near $(M + 1)^{M + 1}/M^M$ (soft edge). These facts particularly suggest, as pointed out in \cite{Akemann-Kieburg-Wei13}, the classical bulk and soft edge universality (via the sine kernel and Airy kernel, respectively) should hold in the bulk and the right edge respectively as in the $M = 1$ case, that is, the complex Wishart ensemble, but new limiting distributions are required to describe the local behavior at the hard edge if $M > 1$. The new limiting distributions, characterized by their limiting correlation kernels, were computed in \cite{Kuijlaars-Zhang14} by taking limit from the integral representation \eqref{Knintegral}. Here we note that the new family of kernels is a generalization of the classical Bessel kernel \cite{Tracy-Widom94b} which is the $M = 1$ case of the family, and they are universal correlation kernels since they also appear in many other random models, including Cauchy-chain matrix models \cite{Bertola-Bothner14,Bertola-Gekhtman-Szmigielski14}, products of Ginibre matrices with inverse ones \cite{Forrester14}, biorthogonal ensembles of Borodin \cite{Borodin99} (as shown in \cite{Kuijlaars-Stivigny14}), etc. However, the conceptually simpler universality results in the bulk and at the right edge turn out to be technically more complicated and have been left open in \cite{Kuijlaars-Zhang14}.

It is the aim of this paper to confirm the bulk and soft edge universality in the products of Ginibre matrices $Y_M$. Our main results are stated in the next section.

\subsection{Statement of the main results}
We start with the definition of the sine kernel (see \cite[Theorem 3.1.1]{Anderson-Guionnet-Zeitouni10}; here we take a different normalization):
\begin{equation}
K_{\sin}(x,y):=\frac{\sin\pi(x-y)}{\pi(x-y)}.
\end{equation}
Recall the correlation kernel $K_n(x,y)$ given in \eqref{Knintegral},  our first result is stated as follows:
\begin{thm}[Bulk universality]\label{thm:bulk}
For $x_0\in(0, (M + 1)^{M + 1}/M^M)$, which is parametrized through \eqref{eq:para x} by  $\varphi = \varphi(x_0) \in (0, \pi/(M+1))$, we have, with $\nu_1,\ldots,\nu_M$ being fixed,
\begin{equation}\label{eq:bulk univ}
 \lim_{n \to \infty} \frac{e^{-\pi\xi\cot\varphi}}{e^{-\pi \eta \cot\varphi}} \frac{n^{M-1}}{\rho(\varphi)} K_n\left(n^M\left(x_0+\frac{\xi}{n\rho(\varphi)}\right), n^M\left(x_0+\frac{\eta}{n\rho(\varphi)}\right)\right) =K_{\sin}(\xi,\eta)
\end{equation}
uniformly for $\xi$ and $\eta$ in any compact subset of $\mathbb{R}$, where $\rho(\varphi)$ is defined in \eqref{eq:density}.
\end{thm}

Next, recall the Airy kernel defined by \cite[Section 3.1]{Anderson-Guionnet-Zeitouni10}
\begin{equation}\label{def:airy kernel}
  K_{\ai}(x,y):=\frac{\ai(x)\ai'(y)-\ai'(x)\ai(y)}{x-y} =\frac{1}{(2\pi i)^2}\int_{\gamma_R}\ud \mu\int_{\gamma_L}\ud \lambda \frac{e^{\frac{\mu^3}{3}-x\mu}}{e^{\frac{\lambda^3}{3}-y\lambda}}\frac{1}{\mu-\lambda},
\end{equation}
where $\gamma_R$ and $\gamma_L$ are symmetric with respect to the imaginary axis,
and $\gamma_R$ is a contour in the right-half plane going from $e^{-\pi i/3}\cdot \infty$ to $e^{\pi i/3}\cdot \infty$;
see Figure \ref{fig:Airy_curve} for an illustration.
\begin{figure}[t]
  \centering
  \begin{overpic}[scale=.6]{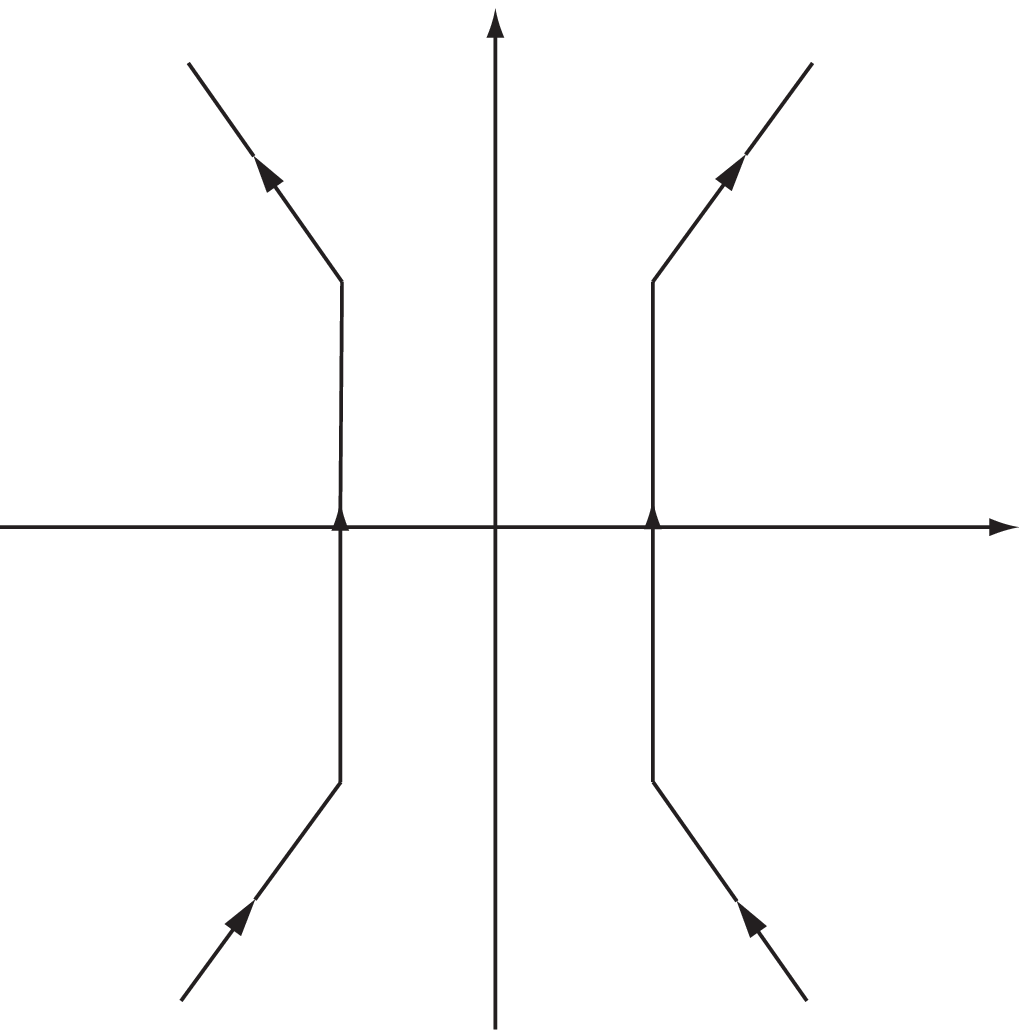}
    \put(15,88){$\gamma_L$} \put(76,88){$\gamma_R$}
  \end{overpic}
  \caption{The contours $\gamma_L$ and $\gamma_R$ in the definition of Airy kernel.}
  \label{fig:Airy_curve}
\end{figure}

\begin{thm}[Soft edge universality]\label{thm:edge}
  With $\nu_1,\ldots,\nu_M$ being fixed, we have
  \begin{equation}\label{eq:edge univer}
    \lim_{n \to \infty} \frac{e^{-2^{-\frac{1}{3}}(M + 1)^{\frac{2}{3}} \xi n^{\frac{1}{3}}}}{e^{-2^{-\frac{1}{3}}(M + 1)^{\frac{2}{3}} \eta n^{\frac{1}{3}}}} n^{M - \frac{2}{3}}c_2 K_n\left(n^M\left( x_* + \frac{ c_2 \xi}{n^{\frac{2}{3}}}\right), n^M\left( x_* + \frac{ c_2 \eta}{n^{\frac{2}{3}}}\right)\right) = K_{\Ai}(\xi, \eta)
  \end{equation}
  uniformly for $\xi$ and $\eta$ in any compact subset of $\mathbb{R}$, where
  \begin{equation}\label{def:c2}
    x_* = \frac{(M + 1)^{M + 1}}{M^M} \quad \text{and} \quad c_2=\frac{(M + 1)^{M + \frac{2}{3}}}{2^{\frac{1}{3}} M^{M - 1}}.
\end{equation}
\end{thm}
Theorems \ref{thm:bulk} and \ref{thm:edge} then imply that the universal scaling limits of the correlation kernel (in the bulk or at the soft edge) that are typical for unitary random matrix ensembles also occur in products of complex Ginibre random matrices.

\begin{rmk}
  If we strengthen the result in Theorem \ref{thm:edge} from uniform convergence into the trace norm convergence of the integral operators with respect to the correlation kernels, then as a direct consequence we have that the limiting distribution of the largest squared singular value, after rescaling, converges to the Tracy-Widom distribution \cite[Theorem 3.1.5]{Anderson-Guionnet-Zeitouni10}. Since the proof of trace norm convergence is only a technical elaboration that confirms a well-expected result, we do not give the detail.
\end{rmk}

\subsection{About the proof}\label{subsec:aboutproof}
Our proof of the main theorems is based on a steepest descent analysis of the double contour integral \eqref{Knintegral}, whose integrand contains products and ratios of gamma functions with large arguments. By Stirling's formula, the logarithms of the gamma functions are approximated by elementary functions for $n$ large, which play the role of phase function. In the bulk regime there are two complex conjugate saddle points, while in the edge regime these two saddle points coalesce into one. The main challenge of the proof is to find suitable contours of integration and sophisticated estimates of integrals. As we shall see later, the parametrization \eqref{eq:para x} will be essential in the analysis. Our strategy can be generalized to deal with some other product models introduced recently, where the correlation kernels have similar structures. We will also discuss about this aspect at the end of this paper.

Here we note that the steepest descent analysis of a double contour integral involving gamma functions was used in a different random matrix model in \cite{Adler-van_Moerbeke-Wang11}, where the limiting Pearcey kernel was derived. The preprint \cite{Forrester-Wang15} that considers a random matrix model similar to that in \cite{Adler-van_Moerbeke-Wang11} applies the method detailed in this paper to perform asymptotic analysis. See also \cite{Zhang15}.

The rest of this paper is organized as follows. Theorems \ref{thm:bulk}--\ref{thm:edge} are proved in Section \ref{sec:proofs}, upon two technical lemmas that are proved in Section \ref{sec:ContConst} on the properties of the specially deformed integral contours. Section \ref{sec:ContConst} also contains the precise construction of the deformed contours. We conclude this paper in Section \ref{sec:extensions} with a discussion on the generalizations of our method to establish similar universal results in other models of products of random matrices. We present two more examples, one is the product of $M$ Ginibre matrices and the inverse of $K$ Ginibre matrices studied by Forrester \cite{Forrester14}, and the other one is the product of $M-1$ Ginibre matrices with one truncated unitary matrix considered by Kuijlaars and Stivigny \cite{Kuijlaars-Stivigny14}.

\section{Proofs of the main theorems}\label{sec:proofs}

\subsection{Notations and contour deformations}\label{subsec:notations}

For notational simplicity, we set
\begin{equation} \label{eq:defn_F}
  F(z; a) := \log \left( \frac{\prod_{j=0}^M \Gamma(z+\nu_j+1)}{ \Gamma(z - n + 1)} a^{-z} \right), \qquad a\geq 0,
\end{equation}
where the logarithm takes the principal branch and we assume that the value of $\log z$ for $z \in (-\infty, 0)$ is continued from above. The asymptotics of $F$ is crucial in our analysis. To proceed, note the Stirling's formula for gamma function \cite[formula 5.11.1]{Boisvert-Clark-Lozier-Olver10} reads
\begin{equation}\label{eq:stirling}
\log \Gamma(z) = \left(z-\frac{1}{2}\right)\log z-z+\frac{1}{2}\log (2\pi)+\mathcal{O}\left(\frac{1}{z}\right)
\end{equation}
as $z\to\infty$ in the sector $\lvert \arg z \rvert \leq \pi-\epsilon $ for some $\epsilon > 0$. Hence it follows that if $\lvert z \rvert \to \infty$ and $\lvert z - n \rvert \to \infty$, while $\arg z$ and $\arg(z - n)$ are in $(-\pi + \epsilon, \pi - \epsilon)$, then uniformly
\begin{multline} \label{eq:F_in_F_tilde}
  F(z; a) = \tilde{F}(z; a) + \sum_{j=0}^M\left(\nu_j+\frac{1}{2}\right) \log z - \frac{1}{2} \log(z - n)\\
   + \frac{M}{2} \log(2\pi) +\bigO(\min(\lvert z \rvert, \lvert z - n \rvert)^{-1}),
\end{multline}
where
\begin{equation} \label{eq:tilde_F_intermidiate}
  \tilde{F}(z; a) = (M + 1)z(\log z - 1) - (z - n)(\log(z - n) - 1) - z\log a.
\end{equation}
Furthermore, we have
\begin{equation} \label{eq:defn_hat_F}
  \tilde{F}(nz; n^M a) = n \hat{F}(z; a) + n\log n,
\end{equation}
where
\begin{equation} \label{eq:hat_F}
  \hat{F}(z; a) = (M + 1)z(\log z - 1) - (z - 1)(\log(z - 1) - 1) - z\log a.
\end{equation}
The behaviour of $\Re \hat{F}(z; a)$ is crucial in the saddle point analysis in this paper, and we plot the level line $\Re \hat{F}(z; a) = \Re \hat{F}(w_{\pm}; a)$ in Figure \ref{fig:level_set}, where $w_{\pm}$ are the critical points of $\hat{F}(z; a)$, see \eqref{eq:defn_w_pm} and \eqref{eq:saddle bulk}.

To prove the results of universality, we also need to deform the contours in \eqref{Knintegral}. First we note that the integral contour for $s$ can be replaced by any infinite contour $\mathcal{C}$ that is taken to go from $-i\infty$ to $i\infty$, as long as $\Sigma$ is on the right side of $\mathcal{C}$. Thus (with shorthand notation $F$ defined in \eqref{eq:defn_F}), we express \eqref{Knintegral} as
\begin{equation} \label{eq:deformed_contour}
  \begin{split}
    K_n(x,y) = & \frac{y^{-1}}{(2\pi i)^2} \int_{\mathcal{C}} \ud s \oint_{\Sigma}  \ud t \frac{e^{F(s; y)}}{e^{F(t; x)}} \frac{1}{s - t}.
  \end{split}
\end{equation}

In the proof of the soft edge universality, we will further deform $\mathcal{C}$ in \eqref{eq:deformed_contour} such that $\Sigma$ is on its left, and it turns out that the resulting double contour integral remains the same. To see this, let $\mathcal{C}$ and $\mathcal{C}'$ be two infinite contours from $-i\infty$ to $i\infty$ such that $\Sigma$ lies between $\mathcal{C}$ and $\mathcal{C}'$, that is, $\Sigma$ is enclosed by $\mathcal{C} \cup \mathcal{C}'$. Applying the residue theorem to the integral on $\mathcal{C} \cup \mathcal{C}'$, it follows
\begin{equation}
  \int_\mathcal{C} \ud s \oint_{\Sigma}  \ud t \frac{e^{F(s; y)}}{e^{F(t; x)}} \frac{1}{s - t} - \int_{\mathcal{C}'} \ud s \oint_{\Sigma}  \ud t \frac{e^{F(s; y)}}{e^{F(t; x)}} \frac{1}{s - t} = 2\pi i \int_{\Sigma} \left( \frac{x}{y} \right)^t dt = 0.
\end{equation}
Hence, the double contour integral does not change if $\mathcal{C}$ is replaced by $\mathcal{C}'$. The deformation of $\mathcal{C}$ is shown in Figure \ref{fig:contour_deformation}.

Similarly, we can show that if $\Sigma$ is split into two disjoint closed counterclockwise contours $\Sigma = \Sigma_1 \cup \Sigma_2$, which jointly enclose poles $0, 1, \ldots, n-1$, and $\mathcal{C}$ is an infinite contour from $-i\infty$ to $i\infty$ such that $\Sigma_1$ is on the left side of $\mathcal{C}$ and $\Sigma_2$ is on the right side of $\mathcal{C}$, the formula \eqref{eq:deformed_contour} is still valid. We will use such kind of contours in the proof of the bulk universality. The deformation of $\Sigma$ is shown in Figure \ref{fig:contour_deformation}.

\begin{figure}[htb]
  \centering
  \includegraphics{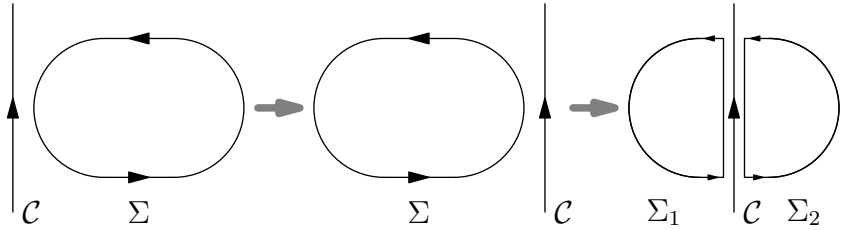}
  \caption{In the contour integral \eqref{eq:deformed_contour}, the position of $\Sigma$ and $\mathcal{C}$ can be switched, and $\Sigma$ can be split into $\Sigma_1$ and $\Sigma_2$ and $\mathcal{C}$ goes between them.}
  \label{fig:contour_deformation}
\end{figure}

To facilitate the asymptotic analysis, throughout the rest of this paper, we shall denote by $D_{r}(a)$ the disc centered at $a$ with radius $r$, and by $\mathbb{C}_{\pm}$ the upper/lower half complex plane, respectively, and if $C$ is a contour in $\compC$ and $r > 0$, then denote by $rC$ the contour $\{ z \in \compC \mid z/r \in C \}$ with the same orientation as $C$.

\subsection{Proof of Theorem \ref{thm:bulk}} \label{subsec:proof_of_bulk_univ}

Before going in the detail of the proof, we sketch the strategy. As mentioned in the end of Section \ref{subsec:notations}, we use the integral representation \eqref{eq:deformed_contour} of the kernel $K_n(x, y)$, with the contours deformed, such that $\Sigma$ is split into two parts, and the contour $\mathcal{C}$ goes between them. But in the asymptotic analysis, we group the ``curved'' part and the ``vertical'' part of $\Sigma$ into $\Sigma_{\curved}$ and $\Sigma_{\vertical}$ respectively; see Figure \ref{fig:Bulk} to get a visual idea of these two parts. Then we compute $K_n(x, y)$ as $I_1 + I_2$, where $I_1$ is the integral \eqref{eq:deformed_contour} over $\mathcal{C}$ and $\Sigma_{\curved}$, while $I_2$ is that over $\mathcal{C}$ and $\Sigma_{\vertical}$. $I_1$ is evaluated by the usual saddle point method, as detailed in Section \ref{subsubsec:I_1}, and it turns out to be the insignificant part; $I_2$ is evaluated by an application of Cauchy's theorem, see Section \ref{subsubsec:I_2}, and it turns out to be the main contribution.

\subsubsection{Exact deformation of the contours}

For any $x_0\in (0, (M + 1)^{M + 1}/M^M)$, we use the parametrization \eqref{eq:para x}, and let $\varphi$ be the unique real number in $(0, \pi/(M + 1))$ such that
\begin{equation*}
  x_0 = \frac{(\sin((M + 1)\varphi))^{M + 1}}{\sin \varphi (\sin(M \varphi))^M}.
\end{equation*}

To prove the bulk universality, we assume that the arguments $x$ and $y$ in \eqref{eq:deformed_contour} are in the form
\begin{equation} \label{eq:defn_xy_bulk}
x = n^M \left(x_0 + \frac{ \xi}{n \rho(\varphi)}\right), \qquad y =n^M \left(x_0 + \frac{\eta}{n \rho(\varphi)}\right),
\end{equation}
where $\xi$ and $\eta$ are in a compact subset of $\realR$ and $\rho(\varphi)$ is given in \eqref{eq:density}. For the asymptotic analysis, we denote
\begin{equation} \label{eq:defn_w_pm}
  w_{\pm} = \frac{\sin((M + 1)\varphi)}{\sin(M\varphi)} e^{\pm i\varphi}.
\end{equation}
They are two saddle points of the function $\hat{F}(z; x_0)$ defined in \eqref{eq:hat_F}, for
\begin{equation}\label{eq:saddle bulk}
  \hat{F}_z(w_\pm; x_0) := \left. \frac{\ud}{\ud z} \hat{F}(z; x_0) \right\rvert_{z = w_{\pm}} = 0.
\end{equation}
It is also straightforward to check
\begin{equation}
  \hat{F}_{zz}(w_{\pm}; x_0) := \left. \frac{\ud^2}{\ud z^2} \hat{F}(z; x_0) \right\rvert_{z = w_{\pm}} = \frac{1}{w_{\pm}}\left(M + 1 - \frac{\sin((M + 1)\varphi)}{\sin \varphi} e^{\mp iM\varphi}\right). \label{eq:second_derivative_bulk}
\end{equation}

\begin{figure}[t]
  \centering
  \begin{overpic}[scale=.6]{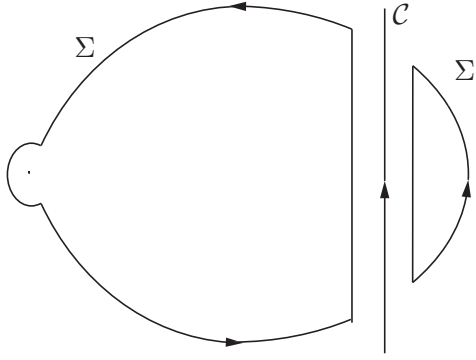}
    \put(15,75){$\Sigma$} \put(82,82){$\mathcal{C}$} \put(95,70){$\Sigma$}
  \end{overpic}
  \caption{The contours $\mathcal{C}$ and $\Sigma$ defined in \eqref{eq:C_contour_bulk} and \eqref{eq:Sigma_contour_bulk}, respectively.}
  \label{fig:Bulk}
\end{figure}
The shapes of the contours $\mathcal{C}$ and $\Sigma$ in \eqref{eq:deformed_contour} used in this subsection are schematically illustrated in Figure \ref{fig:Bulk}, and are precisely described as follows, based on the two contours $\tilde{\mathcal{C}}_{x_0}$ and $\tilde{\Sigma}^{\epsilon}$ explicitly constructed in Section \ref{subsec:contours_bulk}. 
The contour $\mathcal{C}$ is chosen to be
\begin{equation} \label{eq:C_contour_bulk}
  \mathcal{C} = n \tilde{\mathcal{C}}_{x_0},
\end{equation}
\ie, $\mathcal{C}$ is the vertical, upward contour through the points $nw_+$ and $nw_-$. To describe $\Sigma$, we let $\epsilon$ and $\epsilon'$ be two small enough positive constants. Then the contour $\Sigma$ is defined by
\begin{equation} \label{eq:Sigma_contour_bulk}
  \Sigma = \Sigma_{\curved} \cup \Sigma_{\vertical},
\end{equation}
where
\begin{equation}\label{def:sigm_curved_bulk}
\Sigma_{\curved} = \Sigma_1 \cup \Sigma_2,  \qquad \Sigma_{\vertical} = \Sigma_3 \cup \Sigma_4,
\end{equation}
and
\begin{equation} \label{eq:Sigma_contour_bulk_2}
  \begin{gathered}
    \begin{aligned}
      \Sigma_1 = {}& n\tilde{\Sigma}^r \cap \{ z \mid \Re z \leq \Re nw_{\pm} - \epsilon \}, \\
      \Sigma_2 = {}& n\tilde{\Sigma}^r \cap \{ z \mid \Re z \geq \Re nw_{\pm} + \epsilon \}, \\
    \end{aligned}
    \quad \text{with} \quad r = \frac{[\epsilon' n] + \frac{1}{2}}{n}, \\
    \begin{aligned}
      \Sigma_3 = {}& \text{vertical bar connecting the two ending points of $\Sigma_1$}, \\
      \Sigma_4 = {}& \text{vertical bar connecting the two ending points of $\Sigma_2$}.
    \end{aligned}
  \end{gathered}
\end{equation}
Note that $\Sigma$ consists of two disjoint closed contours: $\Sigma_1 \cup \Sigma_3$ and $\Sigma_2 \cup \Sigma_4$, whose orientations are counterclockwise. By the arguments at the end of Section \ref{subsec:notations}, such kind of contour deformation is allowed. Here we assume that $\Re nw_{\pm}$ is not an integer, so that $\mathcal{C}$ does not pass through any integer point. We further assume that $\epsilon$ is small enough so that $\mathcal{C}$, $\Sigma_3$ and $\Sigma_4$ all lie between two consecutive integers $k$ and $k + 1$. In the case that $\Re nw_{\pm} \in \intZ$ and $\mathcal{C}$ passes through an integer point, we simply shift contour $\mathcal{C}$ horizontally by $1/2$ to make it go between two consecutive integers, and all arguments below work in the same way.

In the asymptotic analysis of $K_n(x, y)$, we write \eqref{eq:deformed_contour} as
\begin{equation}
  K_n(x, y) = I_1 + I_2, \quad \text{where} \quad \left.
    \begin{aligned}
      I_1 \\
      I_2
    \end{aligned} \right\} = \frac{y^{-1}}{(2\pi i)^2} \times
  \begin{cases}
    \displaystyle \int_{\mathcal{C}} \ud s \oint_{\Sigma_{\curved}}  \ud t \frac{e^{F(s; y)}}{e^{F(t; x)}} \frac{1}{s - t}, \\
    \displaystyle \int_{\mathcal{C}} \ud s \oint_{\Sigma_{\vertical}}  \ud t \frac{e^{F(s; y)}}{e^{F(t; x)}} \frac{1}{s - t}.
  \end{cases}
\end{equation}

The following properties of $F(z; n^M x_0)$ on the contours $\Sigma_{\curved}$ and $\mathcal{C}$ will play an important role in our later analysis.
\begin{lem} \label{lem:bulk_ineq}
  There exists a constant $\delta > 0$ such that for $n$ large enough,
  \begin{align}
    \Re F(z; n^M x_0) \geq {}& \Re F(nw_{\pm}; n^M x_0) + \delta n \left\lvert \frac{z}{n} - w_{\pm} \right\rvert^2 & & \text{for $z \in \Sigma_{\curved} \cap D_{n^{\frac{3}{5}}}(nw_{\pm})$}, \label{eq:minimum_Sigma_local_F} \\
    \Re F(z; n^M x_0) > {}& \Re F(nw_{\pm}; n^M x_0) + \delta n^{\frac{1}{5}} & & \text{for $z \in \Sigma_{\curved}\setminus \left( D_{n^{\frac{3}{5}}}(nw_{+})\cup D_{n^{\frac{3}{5}}}(nw_{-})\right)$}, \label{eq:minimum_Sigma_F} \\
    \Re F(z; n^M x_0) \leq {}& \Re F(nw_{\pm}; n^M x_0) - \delta n \left\lvert \frac{z}{n} - w_{\pm} \right\rvert^2 & & \text{for $z \in \mathcal{C} \cap D_{n^{\frac{3}{5}}}(nw_{\pm}) $}, \label{eq:maximum_C_local_F} \\
    \Re F(z; n^M x_0) < {}& \Re F(nw_{\pm}; n^M x_0) - \delta n^{\frac{1}{5}} & & \text{for $z \in \mathcal{C}\setminus  \left( D_{n^{\frac{3}{5}}}(nw_{+})\cup D_{n^{\frac{3}{5}}}(nw_{-})\right)$}, \label{eq:maximum_C_F} \\
    \Re F(z; n^M x_0) < {}& \Re F(nw_{\pm}; n^M x_0) - \delta \lvert z \rvert & & \text{for $z \in \mathcal{C}\cap \{ \lvert z \rvert  >\delta^{-1}n \}$}. \label{eq:maximum_C_global_F}
  \end{align}
\end{lem}
Since the proof of this lemma is lengthy and technical, we decide to postpone it to Section \ref{subsec:proof_lem_bulk_ineq}.

To grasp the meaning of Lemma \ref{lem:bulk_ineq} before getting involved in the delicate inequalities, it is better to consult Figure \ref{fig:level_set}, the level line $\Re \hat{F}(z; x_0) = \hat{F}(w_{\pm}; x_0)$. Since the $n \to \infty$ behaviour of $F(z; n^M x_0)$ is determined by $\hat{F}(z/n; x_0)$, as shown in \eqref{eq:F_in_F_tilde}, \eqref{eq:defn_hat_F} and \eqref{eq:hat_F}, by comparing the shapes of $\Sigma_{\curved}$ and $\mathcal{C}$ with the level line in Figure \ref{fig:level_set}, we find that the value $\Re F(z; n^M x_0)$ attains its maximum over $\mathcal{C}$ around $nw_{\pm}$, while it attains its minimum over $\Sigma_{\curved}$ around these two points.

\begin{figure}[htb]
  \centering
  \includegraphics[width=0.5\linewidth]{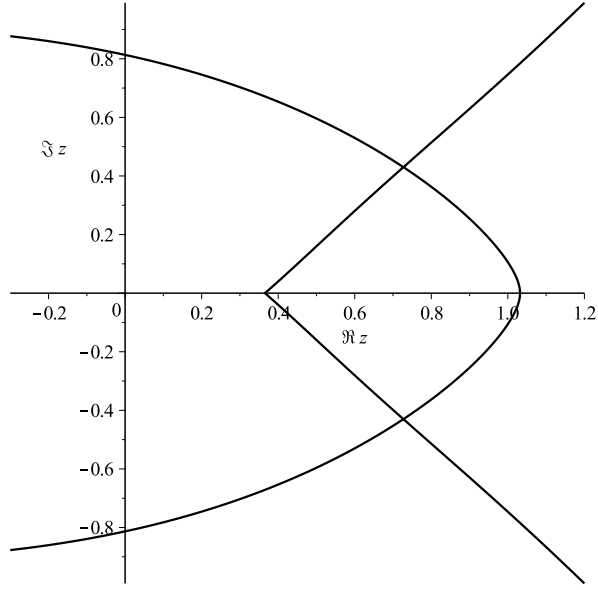}
  \caption{The level line $\Re \hat{F}(z; x_0) = \Re \hat{F}(w_{\pm}; x_0)$ with $M = 3$ and $x_0 = 1$. The level line has two intersections $w_{\pm}$, and it divides the complex plane into five parts. In the top, central, and bottom parts, $\Re \hat{F}(z; x_0) < \Re \hat{F}(w_{\pm}; x_0)$, while in the left and right parts, $\Re \hat{F}(z; x_0) > \Re \hat{F}(w_{\pm}; x_0)$.}
  \label{fig:level_set}
\end{figure}

\subsubsection{Evaluation of $I_2$ as $\epsilon \to 0$} \label{subsubsec:I_2}

Note that the contour $\Sigma$ depends on a parameter $\epsilon$. By taking $\epsilon \to 0$, we have
\begin{equation} \label{eq:formula_I_2}
  \begin{split}
    \lim_{\epsilon \to 0} I_2 := {}& \lim_{\epsilon \to 0} \frac{y^{-1}}{(2\pi i)^2} \int_\mathcal{C} \ud s \int_{\Sigma_3 \cup \Sigma_4} \ud t \frac{e^{F(s; y)}}{e^{F(t; x)}} \frac{1}{s - t} \\
    = {}& \frac{y^{-1}}{2\pi i} \int^{nw_+}_{nw_-} \frac{e^{F(s; y)}}{e^{F(s; x)}}\ud s = \frac{y^{-1}}{2\pi i} \int^{nw_+}_{nw_-} \left( \frac{x}{y} \right)^s \ud s \\
    = {}& \frac{1}{2\pi i y \log(\frac{x}{y})} \left( \left( \frac{x}{y} \right)^{nw_+} - \left( \frac{x}{y} \right)^{nw_-} \right).
  \end{split}
\end{equation}
Here we assume that $x \neq y$, and use Cauchy's theorem in the first step. With the values of $x, y$ given in \eqref{eq:defn_xy_bulk}, it is readily seen that as $n\to\infty$,
\begin{equation}
\frac{x}{y}=1+\frac{\xi-\eta}{n\rho(\varphi)x_0}+\mathcal{O}\left(n^{-2}\right),\quad
y\log\left(\frac{x}{y}\right)=\frac{n^{M-1}(\xi-\eta)}{\rho(\varphi)}\left(1+\mathcal{O}\left(n^{-1}\right)\right),
\end{equation}
where $\varphi$ is related to $x_0$ by \eqref{eq:para x} and $\rho(\varphi)$ is defined in \eqref{eq:density}. These approximations, together with $w_{\pm}$ given in \eqref{eq:defn_w_pm}, imply that if $\xi \neq \eta$, then
\begin{align} \label{eq:result_I_1}
  \lim_{\epsilon \to 0} I_2 = {}& \frac{\rho(\varphi)}{2\pi i n^{M-1}(\xi - \eta)\left(1 + \bigO \left(n^{-1}\right)\right)} \left( e^{\frac{(\xi - \eta)w_+}{\rho(\varphi)x_0}} \left(1 + \bigO \left(n^{-1}\right)\right) - e^{\frac{(\xi - \eta)w_-}{\rho(\varphi)x_0}} \left(1 + \bigO \left(n^{-1}\right)\right) \right)  \nonumber
 \\
 = {}& \frac{e^{\pi \xi \cot\varphi}}{e^{\pi \eta \cot\varphi}} \frac{\rho(\varphi)}{n^{M - 1}} \frac{\sin\pi(\xi - \eta)}{\pi(\xi - \eta)} + \bigO\left(n^{-M}\right)
 \end{align}
for large $n$. Note that although we define $I_2$ as a function with real variables $x$ and $y$, it is also a well defined analytic function if we understand $x$ and $y$ as complex variables. Then \eqref{eq:result_I_1} also holds if $\xi$ and $\eta$ are distinct complex numbers. By analytic continuation we have that \eqref{eq:result_I_1} also holds for $\xi = \eta$, and particularly for the case that they are identical real numbers.

\subsubsection{Evaluation of $I_1$ as $\epsilon \to 0$} \label{subsubsec:I_1}

Parallel to \eqref{eq:formula_I_2}, by taking the limit $\epsilon \to 0$, it follows
\begin{equation} \label{eq:formula_I_1}
  \begin{split}
    \lim_{\epsilon \to 0} I_1 := {}& \lim_{\epsilon \to 0} \frac{y^{-1}}{(2\pi i)^2} \int_\mathcal{C} \ud s \int_{\Sigma_1 \cup \Sigma_2} \ud t \frac{e^{F(s; y)}}{e^{F(t; x)}} \frac{1}{s - t} \\
    = {}& \frac{y^{-1}}{(2\pi i)^2} \lim_{\epsilon \to 0_+} \int_{n\tilde{\Sigma}^r \setminus \big( D_{\epsilon}(nw_+) \cup D_{\epsilon}(nw_-) \big)} \left( \int_\mathcal{C} \ud s \frac{e^{F(s; y)}}{e^{F(t; x)}} \frac{1}{s - t} \right)\ud t \\
    = {}& \frac{y^{-1}}{(2\pi i)^2} \pv \int_{n\tilde{\Sigma}^r} \left( \int_\mathcal{C} \ud s \frac{e^{F(s; y)}}{e^{F(t; x)}} \frac{1}{s - t} \right)\ud t,
  \end{split}
\end{equation}
where $\pv$ means the Cauchy principal value. For the definition and properties of Cauchy principal value for contour integral; see \cite[Section 8.3, Page 191]{Kanwal97}. We remark that the integral with respect to $t$ on $n\tilde{\Sigma}^r$ in \eqref{eq:formula_I_1} is Riemann integrable, but has discontinuities at $nw_{\pm}$, the intersections of $\mathcal{C}$ and $n\tilde{\Sigma}^r$. Thus there is no serious integrability problem in the Cauchy principal value for the integral over $n\tilde{\Sigma}^r$.

To estimate the limit of $I_1$, we define
\begin{equation}
  \mathcal{C}^{\pm}_{\local} =\mathcal{ C} \cap D_{n^{\frac{3}{5}}}(nw_{\pm}), \qquad \Sigma^{\pm}_{\local} = \Sigma \cap D_{n^{\frac{3}{5}}}(nw_{\pm}),
\end{equation}
and show that the main contribution to the Cauchy principal integral is from $\mathcal{C}^+_{\local} \times \Sigma^+_{\local}$ and $\mathcal{C}^-_{\local} \times \Sigma^-_{\local}$ in the sense that  remaining part of the integral is negligible in the asymptotic analysis.

It is clear that for $s \in \mathcal{C}^+_{\local}$ and $t \in \Sigma^+_{\local}$, we can approximate $F(s; n^M x_0)$ and $F(t; n^M x_0)$ by $\tilde{F}$ as in \eqref{eq:F_in_F_tilde} and furthermore by $\hat{F}$ that is defined in \eqref{eq:hat_F}. We make the change of variables
\begin{equation} \label{eq:relation_st_uv}
  s = nw_+ + n^{\frac{1}{2}} u, \qquad t = nw_+ + n^{\frac{1}{2}} v.
\end{equation}
It then follows from \eqref{eq:defn_xy_bulk}, \eqref{eq:saddle bulk} and \eqref{eq:second_derivative_bulk} that, uniformly for all $s \in D_{n^{3/5}}(nw_+)$,
\begin{equation} \label{eq:asy_e^{F(s;y)}_at_w_+}
  \begin{split}
    e^{F(s; y)} = {}& e^{F\left(s; n^M x_0\right)}\left (1 + \frac{\eta}{n\rho(\varphi) x_0}\right)^{-s} \\
      = {}& n^n e^{\tilde{c}_M}e^{n \hat{F}\left(w_+ + n^{-\frac{1}{2}} u; x_0\right)}\left (1 + \frac{\eta}{n\rho(\varphi) x_0}\right)^{-s} \left(1 + \bigO \left (n^{-\frac{1}{2}}\right) \right) \\
      = {}& n^n e^{\tilde{c}_M+ n\hat{F}\left(w_+; x_0\right)} e^{\frac{\hat{F}_{zz}(w_+; x_0)}{2} u^2}\left (1 + \frac{ \eta}{n \rho(\varphi) x_0}\right)^{-s} \left(1 + \bigO\left(n^{-\frac{1}{5}}\right)\right) \\
      = {}& n^n e^{\tilde{c}_M+n \hat{F}(w_+; x_0)} e^{\frac{\hat{F}_{zz}(w_+; x_0)}{2} u^2} e^{-\frac{w_+ \eta}{x_0\rho(\varphi)} } \left(1 + \bigO\left(n^{-\frac{1}{5}}\right)\right),
  \end{split}
\end{equation}
where
\begin{equation}
\tilde c_M=\sum_{j=1}^M\left(\nu_j+\frac{1}{2}\right) \log(nw_+) + \frac{1}{2}\log\left(\frac{w_+}{1-w_+}\right)\\
   + \frac{M}{2} \log(2\pi).
\end{equation}
A parallel argument yields that uniformly for $t \in D_{n^{3/5}}(nw_+)$,
\begin{equation} \label{eq:asy_e^{F(t;x)}_at_w_+}
  e^{F(t; x)} = n^n e^{\tilde{c}_M+n \hat{F}(w_+; x_0)} e^{\frac{\hat{F}_{zz}(w_+; x_0)}{2} v^2} e^{-\frac{ w_+ \xi}{x_0\rho(\varphi)}} \left(1 + \bigO\left(n^{-\frac{1}{5}}\right)\right).
\end{equation}
As a consequence, (noting that $s - t = \sqrt{n} (u - v)$)
\begin{multline} \label{eq:crossing_estimate}
  \pv \int_{\mathcal{C}^+_{\local}} \ud s \oint_{\Sigma^+_{\local}}  \ud t \frac{e^{F(s; y)}}{e^{F(t; x)}} \frac{1}{s - t} \\
  = \frac{e^{-\frac{w_+(\xi - \eta)}{\rho(\varphi)x_0} }}{\sqrt{n}} \pv \int_{\mathcal{C}^+_{\local}} \ud s \oint_{\Sigma^+_{\local}}  \ud t \frac{e^{\frac{\hat{F}_{zz}(w_+; x_0)}{2} u^2}}{e^{\frac{\hat{F}_{zz}(w_+; x_0)}{2} v^2}} \frac{1 + \bigO(n^{-\frac{1}{5}})}{u - v},
\end{multline}
where on the right-hand side, we understand $u$ and $v$ as functions of $s$ and $t$ respectively, as defined by \eqref{eq:relation_st_uv}. Note that the $\bigO(n^{-1/5})$ term in the integrand on the right-hand side of \eqref{eq:crossing_estimate} is uniform and analytic in $D_{n^{3/5}}(nw_+)$. Comparing the result of \eqref{eq:asy_e^{F(s;y)}_at_w_+} with $y = n^M x_0$ and Lemma \ref{lem:bulk_ineq}, we have that there exists a constant $\epsilon_1 > 0$ such that for all $s \in \mathcal{C}^+_{\local}$,
\begin{equation} \label{eq:descent_direction}
  \left\lvert e^{\frac{\hat{F}_{zz}(w_+; x_0)}{2} u^2} \right \rvert \leq e^{-\epsilon_1 \lvert u \rvert^2}.
\end{equation}
Similarly, a comparison between \eqref{eq:asy_e^{F(t;x)}_at_w_+} and \eqref{eq:minimum_Sigma_local_F} in Lemma \ref{lem:bulk_ineq} implies that there is a constant $\epsilon_2 > 0$ such that for all $t \in \Sigma^+_{\local}$,
\begin{equation} \label{eq:ascent_direction}
  \left  \lvert e^{\frac{\hat{F}_{zz}(w_+; x_0)}{2} v^2}\right \rvert \geq e^{\epsilon_2 \lvert v \rvert^2}.
\end{equation}
Hence a standard application of the saddle point method yields
\begin{multline} \label{eq:bulk_est_insenssial_1}
  \pv \int_{\mathcal{C}^+_{\local}} \ud s \oint_{\Sigma^+_{\local}}  \ud t \frac{e^{F(s; y)}}{e^{F(t; x)}} \frac{1}{s - t} \\
  = \lim_{\epsilon \to 0_+} \int_{\mathcal{C}^+_{\local}} \ud s \oint_{\Sigma^+_{\local} \setminus D_{\epsilon}(nw_+)}  \ud t \frac{e^{F(s; y)}}{e^{F(t; x)}} \frac{1}{s - t} = \bigO\left(n^{\frac{1}{2}}\right).
\end{multline}
In a similar manner, by setting
\begin{equation}
  s = nw_- + n^{\frac{1}{2}} u, \quad t = nw_- + n^{\frac{1}{2}} v,
\end{equation}
 we have
\begin{equation} \label{eq:bulk_est_insenssial_2}
  \pv \int_{\mathcal{C}^-_{\local}} \ud s \oint_{\Sigma^-_{\local}}  \ud t \frac{e^{F(s; y)}}{e^{F(t; x)}} \frac{1}{s - t} = \bigO\left(n^{\frac{1}{2}}\right).
\end{equation}

Finally, we note by \eqref{eq:minimum_Sigma_F}, \eqref{eq:maximum_C_F} and \eqref{eq:maximum_C_global_F} in Lemma \ref{lem:bulk_ineq} that there exists $\epsilon_3 > 0$ such that for large enough $n$
\begin{align}
  \left\lvert e^{-F(t; x)} \right\rvert = \left\lvert e^{-F\left(t; n^M x_0\right)} \left( 1 + \frac{\xi}{n \rho(\varphi)x_0} \right)^t \right\rvert < {}& \left\lvert e^{-F\left(nw_{\pm}; n^Mx_0\right)} \right\rvert e^{-\epsilon_3 n^{\frac{1}{5}}} \quad \text{if $t \in \Sigma \setminus \Sigma^{\pm}_{\local}$}, \label{eq:est_e^F(t;x)_bulk_outer} \\
  \left\lvert e^{F(s;y)} \right\rvert = \left\lvert e^{F\left(s; n^M x_0\right)} \left( 1 + \frac{ \eta}{n \rho(\varphi) x_0} \right)^{-s} \right\rvert < {}&
  \begin{cases}
    \left\lvert e^{F\left(nw_{\pm}; n^Mx_0\right)} \right\rvert e^{-\epsilon_3 n^{\frac{1}{5}}} & \text{if $s \in \mathcal{C} \setminus \mathcal{C}^{\pm}_{\local}$},
    \vspace{1mm} \\
    \left\lvert e^{F\left(nw_{\pm}; n^Mx_0\right)} \right \rvert e^{-\epsilon_3 \lvert s \rvert} & \text{if $s \in \mathcal{C}\cap \{ \lvert s \rvert > \frac{n}{\epsilon_3} \}$.}
  \end{cases} \label{eq:est_e^F(s;y)_bulk_outer}
\end{align}
With the aid of the estimates \eqref{eq:est_e^F(t;x)_bulk_outer}, \eqref{eq:est_e^F(s;y)_bulk_outer}, \eqref{eq:asy_e^{F(s;y)}_at_w_+}, \eqref{eq:asy_e^{F(t;x)}_at_w_+}, \eqref{eq:descent_direction} and \eqref{eq:ascent_direction}, we obtain
\begin{multline} \label{eq:bulk_est_insenssial_3}
  \pv \int_\mathcal{C} \ud s \oint_{\Sigma}  \ud t \frac{e^{F(s; y)}}{e^{F(t; x)}} \frac{1}{s - t} - \pv \int_{\mathcal{C}^+_{\local}} \ud s \oint_{\Sigma^+_{\local}}  \ud t \frac{e^{F(s; y)}}{e^{F(t; x)}} \frac{1}{s - t} \\
  - \pv \int_{\mathcal{C}^-_{\local}} \ud s \oint_{\Sigma^-_{\local}}  \ud t \frac{e^{F(s; y)}}{e^{F(t; x)}} \frac{1}{s - t} = \bigO\left(e^{-\epsilon n^{\frac{1}{5}}}\right).
\end{multline}
This, together with \eqref{eq:bulk_est_insenssial_1}, \eqref{eq:bulk_est_insenssial_2} and \eqref{eq:formula_I_1}, implies
\begin{equation} \label{eq:result_I_2}
  \lim_{\epsilon \to 0} I_1 = y^{-1} \left( \bigO(n^{\frac{1}{2}} + n^{\frac{1}{2}} + \bigO(e^{-\epsilon n^{\frac{1}{5}}}) \right) = \bigO\left(n^{-M + \frac{1}{2}}\right),
\end{equation}
where we use that $y = \bigO(n^M)$. Summing up \eqref{eq:result_I_1} and \eqref{eq:result_I_2} and letting $n \to \infty$, we derive \eqref{eq:bulk univ} and complete the proof of Theorem \ref{thm:bulk}.

\subsection{Proof of Theorem \ref{thm:edge}} \label{subsec:proof_edge}
In view of the scalings of $x,y$ in \eqref{eq:edge univer}, we set
\begin{equation} \label{eq:scaling_of_xy_Airy}
x = n^M \left(x_\ast + \frac{c_2 \xi}{n^{2/3}}\right), \quad y = n^M \left(x_\ast + \frac{c_2 \eta}{n^{2/3}}\right),
\end{equation}
where $\xi,\eta\in \mathbb{R}$,
\begin{equation*}
  x_\ast = \frac{(M + 1)^{M + 1}}{M^M} \quad \text{and} \quad c_2=\frac{(M + 1)^{M + \frac{2}{3}}}{2^{\frac{1}{3}} M^{M - 1}},
\end{equation*}
are defined in \eqref{def:c2}. Thus, we write \eqref{eq:deformed_contour} as
\begin{equation} \label{eq:double_contour_for_Airy}
  K_n(x,y) = \frac{y^{-1}}{(2\pi i)^2} \int_\mathcal{C} \ud s \oint_{\Sigma}  \ud t \frac{e^{F(s; n^M x_\ast)} (1 + n^{-\frac{2}{3}} c^{-1}_1 \eta)^{-s}}{e^{F(t; n^M x_\ast)} (1 + n^{-\frac{2}{3}} c^{-1}_1 \xi)^{-t}} \frac{1}{s - t}
\end{equation}
with
\begin{equation}\label{def:c1}
c_1=\frac{x_\ast}{c_2}=\frac{2^{\frac{1}{3}} (M + 1)^{\frac{1}{3}}}{M}.
\end{equation}

In this case, we will choose the contours $\mathcal{C}$ and $\Sigma$ in \eqref{eq:deformed_contour} such that $\Sigma$ is on the left hand side of $\mathcal{C}$, as illustrated in Figure \ref{fig:SoftContour}.
\begin{figure}
  \centering
  \begin{overpic}[scale=.6]{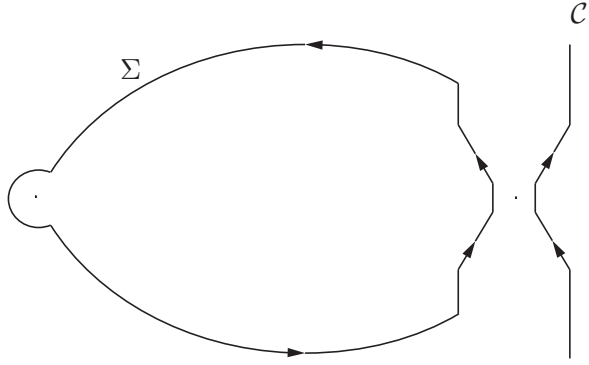}
    \put(20,50){$\Sigma$} \put(100,60){$\mathcal{C}$}
  \end{overpic}
  \caption{The contours $\mathcal{C}$ and $\Sigma$ defined in \eqref{def:Csoft} and \eqref{eq:sigma_soft}}
  \label{fig:SoftContour}
\end{figure}
To describe $\mathcal{C}$, we denote
\begin{equation}\label{def:z0}
z_0 = 1 + \frac{1}{M},
\end{equation}
and then define
\begin{equation}\label{def:Csoft}
  \mathcal{C} = \mathcal{C}_{\local} \cup \mathcal{C}_{\glob},
\end{equation}
where
\begin{multline}\label{eq:clocal}
  \mathcal{C}_{\local} = \left\{ nz_0 + c_1 n^{\frac{2}{3}} re^{\pi i/3} ~\Big{|}~ r \in \left[1, n^{\frac{1}{30}}\right] \right\} \cup \left\{ nz_0 + c_1 n^{\frac{2}{3}} re^{-\pi i/3} ~\Big{|}~ r \in \left[1, n^{\frac{1}{30}}\right] \right\} \\
  \cup \left\{ nz_0 + \frac{c_1 n^{\frac{2}{3}}}{2} + ic_1 n^{\frac{2}{3}} r ~\Big{|}~ r \in \left[-\frac{\sqrt{3}}{2}, \frac{\sqrt{3}}{2}\right] \right\},
\end{multline}
and
\begin{equation}
  \mathcal{C}_{\glob} = \left\{ nz_0 + \frac{1}{2} c_1 n^{\frac{7}{10}} + iy ~\Big{|}~ y \in \left(-\infty, -\frac{\sqrt{3}}{2} c_1 n^{\frac{7}{10}}\right] \cup \left[\frac{\sqrt{3}}{2} c_1 n^{\frac{7}{10}}, \infty \right) \right\}.
\end{equation}
The orientation of $\mathcal{C}$ is taken to be upward. The contour $\Sigma$ is defined as the union of contours
\begin{equation}\label{eq:sigma_soft}
  \Sigma = \Sigma_{\local} \cup \Sigma_{\glob}, \quad \text{and} \quad \Sigma_{\glob} = \Sigma_{\curved} \cup \Sigma_{\vertical}.
\end{equation}
The contour $\Sigma_{\local}$ is defined by
\begin{multline}\label{eq:sigmlocal}
    \Sigma_{\local} = \left\{ nz_0 + c_1 n^{\frac{2}{3}} re^{2\pi i/3} ~\Big{|}~ r \in \left[1, n^{\frac{1}{30}}\right] \right\} \cup \left\{ nz_0 + c_1 n^{\frac{2}{3}} re^{-2\pi i/3} ~\Big{|}~ r \in \left[1, n^{\frac{1}{30}}\right] \right\} \\
    \cup \left\{ nz_0 - \frac{c_1 n^{\frac{2}{3}}}{2} + ic_1 n^{\frac{2}{3}} r ~\Big{|}~ r \in \left [-\frac{\sqrt{3}}{2}, \frac{\sqrt{3}}{2}\right] \right\}.
\end{multline}
The contour $\Sigma_{\glob}$ depends on a small constant $\epsilon > 0$. Define
\begin{equation}
r = \frac{[\epsilon n] + \frac{1}{2}}{n}.
\end{equation}
With the contour $\tilde{\Sigma}^r$ constructed in Section \ref{subsec:contours_bulk}, we denote by $z_{\pm}\in\mathbb{C}_{\pm}$ the two intersection points of $\tilde{\Sigma}^r$ with the vertical line $\Re z = z_0 - \frac{1}{2} c_1 n^{-3/10}$. We then define
\begin{equation}
  \Sigma_{\curved} = n\tilde{\Sigma}^r \cap \left\{\Re z < nz_0 - \frac{1}{2} c_1 n^{\frac{7}{10}} \right\}
\end{equation}
and
\begin{equation}
   \Sigma_{\vertical} = \text{two vertical line segments connecting $nz_{\pm}$ and $nz_0 \pm c_1 n^{\frac{7}{10}} e^{2\pi i/3}$}.
\end{equation}
Note that $\Sigma$ is a closed contour with counterclockwise orientation.

Similar to Lemma \ref{lem:bulk_ineq}, we have the following properties of $F(z; n^M x_*)$ on the contours $\mathcal{C}$ and $\Sigma$.
\begin{lem} \label{lem:edge_ineq}
  There exists a positive constant $\delta > 0$ such that for $n$ large enough,
  \begin{align}
    \Re F(z; n^M x_*) > {}& \Re F(nz_0; n^M x_*) + \delta n^{\frac{1}{10}} & & \text{for $z \in \Sigma_{\glob}$}, \label{eq:F_on_Sigma_edge} \\
    \Re F(z; n^M x_*) < {}& \Re F(nz_0; n^M x_*) - \delta n^{\frac{1}{10}} & & \text{for $z \in \mathcal{C}_{\glob}$}, \label{eq:F_on_C_edge} \\
    \Re F(z; n^M x_*) < {}& \Re F(nz_0; n^M x_*) - \delta \lvert z \rvert & & \text{for $z \in \mathcal{C}_{\glob} \cap \{ \lvert z \rvert > \delta^{-1} n \}$}. \label{eq:F_on_C_infty_edge}
  \end{align}
\end{lem}
The proof of this lemma is postponed to Section \ref{subsec:proof_lem_edge_ineq}.

The strategy now is first to consider the double contour integral in \eqref{eq:double_contour_for_Airy} with $\mathcal{C}$ and $\Sigma$ restricted to $\mathcal{C}_{\local}$ and $\Sigma_{\local}$, respectively. It turns out that the integral with the restricted domain yields the Airy kernel in the large $n$ limit. Later we show that the remaining part of the integral is negligible in the asymptotic analysis.

For $s \in \mathcal{C}_{\local}$ and $t \in \Sigma_{\local}$, we can approximate $F(s; n^M x_\ast)$ and $F(t; n^M x_\ast)$ by $\tilde{F}$ as in \eqref{eq:F_in_F_tilde} and furthermore by $\hat{F}$ that is defined in \eqref{eq:hat_F}. By making the change of variables
\begin{equation} \label{eq:parametrization_Airy}
  s = nz_0 + n^{\frac{2}{3}} c_1 u, \quad t = nz_0 + n^{\frac{2}{3}} c_1 v,
\end{equation}
it follows that
\begin{align} \label{eq:F_tilde_in_F_hat}
  F(s; n^M x_\ast) ={}& \tilde{F}(s; n^M x_\ast) + c_M+\bigO(n^{-\frac{1}{3}})  \nonumber \\
={}&n \hat{F}(z_0 + n^{-\frac{1}{3}}c_1u; x_\ast) + n\log n +c_M+ \bigO(n^{-\frac{1}{3}}),
\end{align}
where
\begin{equation}
  c_M=\sum_{j=1}^M\left(\nu_j+\frac{1}{2}\right) \log\left (\frac{n(M+1)}{M}\right) + \frac{1}{2}\log\left(M+1\right)\\
  + \frac{M}{2} \log(2\pi).
\end{equation}

Straightforward calculations show that
\begin{equation}
  \hat{F}_z(z_0; x_\ast) = 0, \quad \hat{F}_{zz}(z_0; x_\ast) = 0, \quad \hat{F}_{zzz}(z_0; x_\ast) = \frac{M^3}{M + 1}.
\end{equation}
We then obtain from Taylor's expansion of \eqref{eq:hat_F} that
\begin{align} \label{eq:F_hat_Taylor_expansion}
    & \hat{F}(z_0 + n^{-\frac{1}{3}}c_1u; x_\ast )  \nonumber \\
    = {}& \hat{F}(z_0; x_\ast) + \hat{F}_z(z_0; x_\ast)c_1u n^{-\frac{1}{3}} + \frac{1}{2} \hat{F}_{zz}(z_0; x_\ast)c^2_1u^2 n^{-\frac{2}{3}} + \frac{1}{6} \hat{F}_{zzz}(z_0; x_\ast)c^3_1u^3 n^{-1} + \bigO\left(n^{-\frac{6}{5}}\right)  \nonumber \\
    = {}& \hat{F}(z_0; x_\ast) + \frac{u^3}{3n} + \bigO\left(n^{-\frac{6}{5}}\right),
\end{align}
uniformly valid for $u \in D_{n^{1/30}}(0)$. We also note that, by \eqref{def:c1}, \eqref{eq:parametrization_Airy} and \eqref{def:z0},
\begin{equation} \label{eq:simpler_Taylor_Airy}
  \left(1 +n^{-\frac{2}{3}} c^{-1}_1 \eta\right)^{-s} = e^{- 2^{-\frac{1}{3}} (M + 1)^{\frac{2}{3}} \eta n^{\frac{1}{3}}} e^{-u\eta}\left (1 + \bigO\left(n^{-\frac{1}{3}}\right)\right),
\end{equation}
for all $u \in D_{n^{1/30}}(0)$ and $\eta$ in a compact subset of $\mathbb{R}$.
Combining \eqref{eq:F_tilde_in_F_hat}, \eqref{eq:F_hat_Taylor_expansion} and \eqref{eq:simpler_Taylor_Airy}, we find
\begin{equation} \label{eq:asy_F(s;y)}
  e^{F(s; n^M x_\ast)}\left(1 + n^{-\frac{2}{3}} c^{-1}_1 \eta\right)^{-s} = n^ne^{c_M+\hat{F}(z_0; x_\ast)n}e^{\frac{1}{3}u^3 - u\eta} e^{-2^{-\frac{1}{3}}(M + 1)^{\frac{2}{3}} \eta n^{\frac{1}{3}}} \left (1 + \bigO\left(n^{-\frac{1}{5}}\right)
  \right),
\end{equation}
uniformly for  $s \in \mathcal{C}_{\local}$ and $\eta$ in a compact subset of $\mathbb{R}$. Similarly, if $x$ and $t$ are expressed respectively by $\xi$ and $v$ via \eqref{eq:scaling_of_xy_Airy} and \eqref{eq:parametrization_Airy}, where $\xi$ belongs to a compact subset of $\mathbb{R}$ and $t \in \Sigma_{\local}$, we have that uniformly in $t$ and $\xi$
\begin{equation} \label{eq:asy_F(t;x)}
  e^{F(t; n^M x_\ast)}\left(1 + n^{-\frac{2}{3}} c^{-1}_1 \xi \right)^{-s} = n^ne^{c_M+\hat{F}(z_0; x_\ast)n}e^{\frac{1}{3}v^3 - v\xi} e^{-2^{-\frac{1}{3}}(M + 1)^{\frac{2}{3}} \xi n^{\frac{1}{3}}} \left (1 + \bigO\left(n^{-\frac{1}{5}}\right)
  \right).
\end{equation}

Substituting \eqref{eq:asy_F(s;y)} and \eqref{eq:asy_F(t;x)} into the integrand of \eqref{eq:double_contour_for_Airy}, we have
\begin{align} \label{eq:local_asy_analysis_Airy}
  & \frac{y^{-1}}{(2\pi i)^2} \int_{\mathcal{C}_{\local}} \ud s \oint_{\Sigma_{\local}}  \ud t \frac{e^{F(s; n^M x_\ast)}\left (1 + n^{-\frac{2}{3}} c^{-1}_1 \eta\right)^{-s}}{e^{F(t; n^M x_\ast)} \left(1 + n^{-\frac{2}{3}} c^{-1}_1 \xi\right)^{-t}} \frac{1}{s - t} \nonumber \\
   = {}& \frac{e^{2^{-\frac{1}{3}}(M + 1)^{\frac{2}{3}} (\xi - \eta) n^{\frac{1}{3}}}}{n^{M - \frac{2}{3}} x_\ast c^{-1}_1} \left( \frac{1}{(2\pi i)^2} \int_{\mathcal{C}_{r}} \ud u \int_{\Sigma_{r}} \ud v \frac{e^{ \frac{1}{3}u^3 - u\eta}}{e^{\frac{1}{3}v^3 - v\xi}} \frac{1}{u - v} + \bigO\left(n^{-\frac{1}{5}}\right) \right) \nonumber \\
    = {}& \frac{e^{2^{-\frac{1}{3}}(M + 1)^{\frac{2}{3}} (\xi - \eta) n^{\frac{1}{3}}}}{n^{M - \frac{2}{3}}c_2} \left( K_{\Ai}(\xi, \eta) + \bigO\left(n^{-\frac{1}{5}}\right) \right),
 \end{align}
where  $\Sigma_r$ and $\mathcal{C}_r$ are the images of $\mathcal{C}_{\local}$ and $\Sigma_{\local}$ (see \eqref{eq:clocal} and \eqref{eq:sigmlocal}) under the change of variables \eqref{eq:parametrization_Airy} (see Figure \ref{fig:Airy_contours} for an illustration), and the last equality follows from the integral representation of Airy kernel shown in \eqref{def:airy kernel}.

\begin{figure}[h]
  \centering
  \begin{overpic}
    [scale=.6]{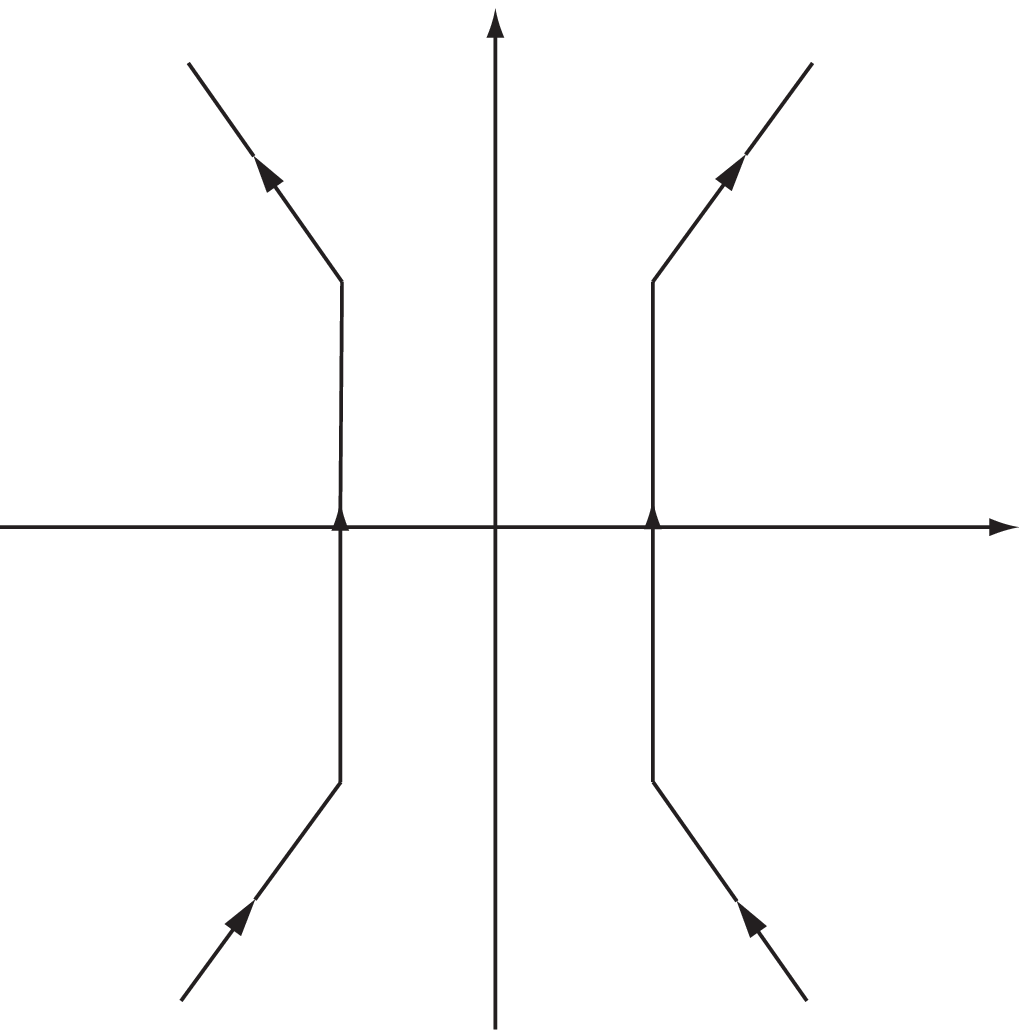}
    \put(11,95){$\Sigma_r$} \put(78,95){$\mathcal{C}_r$}
  \end{overpic}
  \caption{The contours $\Sigma_r$ and $\mathcal{C}_r$ in \eqref{eq:local_asy_analysis_Airy}}
  \label{fig:Airy_contours}
\end{figure}

In a manner similar to \eqref{eq:simpler_Taylor_Airy}, we find that
\begin{multline} \label{eq:peripherial_Airy}
  e^{2^{-\frac{1}{3}}(M + 1)^{\frac{2}{3}} (\eta - \xi) n^{\frac{1}{3}}} \frac{e^{F(s; n^M x_\ast)} \left(1 + n^{-\frac{2}{3}} c^{-1}_1 \eta\right)^{-s}}{e^{F(t; n^M x_\ast)} \left(1 + n^{-\frac{2}{3}} c^{-1}_1 \xi\right)^{-t}} \\= \frac{e^{F(s; n^M x_\ast)} \left(1 + n^{-\frac{2}{3}} c^{-1}_1 \eta\right)^{-(s - nz_0)}}{e^{F(t; n^M x_\ast)} \left(1 + n^{-\frac{2}{3}} c^{-1}_1 \xi\right)^{-(t - nz_0)}} \left(1 + \bigO\left(n^{-\frac{1}{3}}\right)\right).
\end{multline}
Then as a consequence of Lemma \ref{lem:edge_ineq}, there exists a constant $\delta > 0$ such that for $n$ large enough
\begin{align}
  \left\lvert e^{-F\left(t; n^Mx_\ast\right)} \left(1 + n^{-\frac{2}{3}} c^{-1}_1 \xi\right)^{(t - nz_0)} \right\rvert < {}& \left \lvert e^{-F\left(nz_0; n^Mx_\ast\right)} \right\rvert e^{-\delta n^{\frac{1}{10}}} \quad \text{if $t \in \Sigma_{\glob}$}, \label{eq:peripherial_Airy_est_1} \\
  \left\lvert e^{F(s; n^Mx_\ast)} (1 + n^{-\frac{2}{3}} c^{-1}_1 \eta)^{-(s - nz_0)} \right\rvert < {}&
  \begin{cases}
    \left\lvert e^{F(nz_0; n^Mx_\ast)} \right\rvert e^{-\delta n^{\frac{1}{10}}} & \text{if $s \in \mathcal{C}_{\glob}$}, \vspace{1mm} \\
    \left\lvert e^{F(nz_0; n^Mx_\ast)} \right\rvert e^{-\delta \lvert s \rvert} & \text{if $s \in \mathcal{C}_{\glob}\cap \{ \lvert s \rvert > \frac{\delta}{n} \}$}.
  \end{cases} \label{eq:peripherial_Airy_est_2}
\end{align}
We conclude by \eqref{eq:peripherial_Airy}, \eqref{eq:peripherial_Airy_est_1}, \eqref{eq:peripherial_Airy_est_2}, and the asymptotics of the integrand of \eqref{eq:double_contour_for_Airy} given in \eqref{eq:asy_F(s;y)} and \eqref{eq:asy_F(t;x)} that
\begin{multline} \label{eq:contral_of_outer_integral}
  \frac{y^{-1}}{(2\pi i)^2} \int_\mathcal{C} \ud s \oint_{\Sigma}  \ud t \frac{e^{F(s; x_\ast)} (1 + n^{-\frac{2}{3}} c^{-1}_1 \eta)^{-s}}{e^{F(t; x_\ast)} (1 + n^{-\frac{2}{3}} c^{-1}_1 \xi)^{-t}} \frac{1}{s - t} \\
  - \frac{y^{-1}}{(2\pi i)^2} \int_{\mathcal{C}_{\local}} \ud s \oint_{\Sigma_{\local}}  \ud t \frac{e^{F(s; x_\ast)} (1 + n^{-\frac{2}{3}} c^{-1}_1 \eta)^{-s}}{e^{F(t; x_\ast)} (1 + n^{-\frac{2}{3}} c^{-1}_1 \xi)^{-t}} \frac{1}{s - t} = \bigO\left(e^{-\delta n^{\frac{1}{10}}}\right).
\end{multline}

A combination of the above formula and \eqref{eq:local_asy_analysis_Airy} gives us \eqref{eq:edge univer}, and completes the proof of Theorem \ref{thm:edge}.


\section{Contour constructions and proofs of Lemmas \ref{lem:bulk_ineq} and \ref{lem:edge_ineq}}\label{sec:ContConst}

In this section, we first construct two contours $\tilde{\mathcal{C}}_{x_0}$ and $\tilde{\Sigma}^{\epsilon}$, from which we can describe precisely the contours of the double integral \eqref{eq:deformed_contour} used in the proofs of our main theorems. The contour $\tilde{\mathcal{C}}_{x_0}$ depends on $x_0 \in (0, x_*)$, where $x_* = (M + 1)^{M + 1}/M^M$ is defined in \eqref{def:c2}. The other contour $\tilde{\Sigma}^{\epsilon}$ is dependent on a small parameter $\epsilon > 0$. Two technical lemmas regarding the behavior of the function $\Re \hat{F}$ on these two contours are then proved. With the aid of these two lemmas, we finally finish the proofs of Lemmas \ref{lem:bulk_ineq} and \ref{lem:edge_ineq} used in Sections \ref{subsec:proof_of_bulk_univ} and \ref{subsec:proof_edge}, respectively.

\subsection{Constructions of contours $\tilde{\mathcal{C}}_{x_0}$ and $\tilde{\Sigma}^{\epsilon}$} \label{subsec:contours_bulk}


Recall that for each $x_0 \in (0, x_*)$, which can be parametrized by $\varphi \in (0, \pi/(M + 1))$ as in \eqref{eq:para x}, we have two complex conjugate saddle points $w_\pm$ of $\hat{F}(z;x_0)$ defined in \eqref{eq:defn_w_pm}. The contour $\tilde{\mathcal{C}}_{x_0}$ is defined by
\begin{equation}
  \tilde{\mathcal{C}}_{x_0} := \{ z \in \compC \mid \Re z = \Re w_{+}= \Re w_{-}\},
\end{equation}
i.e., a vertical line passing through $\Re w_{\pm}$.

For the construction of $\tilde{\Sigma}^{\epsilon}$, we first define
\begin{equation}
  \tilde{\Sigma} := \tilde{\Sigma}_+ \cup \tilde{\Sigma}_-,
\end{equation}
where
\begin{equation}
  \tilde{\Sigma}_+ := \left\{ \zeta(\phi) ~\Big{|}~ \phi \in \left[0, \frac{\pi}{M + 1}\right] \right\}, \quad \tilde{\Sigma}_- :=\left \{ \overline{\zeta(\phi)} ~\Big{|}~ \phi \in\left [0, \frac{\pi}{M + 1}\right] \right \},
\end{equation}
with
\begin{equation} \label{eq:parametrisation_of_zeta}
  \zeta(\phi) = \frac{\sin((M + 1)\phi)}{\sin(M\phi)}e^{ i\phi}.
\end{equation}
It is easy to check that $\tilde{\Sigma}_{\pm}$ lies in $\compC_{\pm}$, passes through $w_\pm$, and intersects the real line only at $0$ when $\phi = \pi/(M + 1)$, and at $1 + M^{-1}$ when $\phi = 0$. Furthermore, as $\phi$ runs from $0$ to $\pi/(M + 1)$, the value of $|\zeta(\phi)|=\sin((M + 1)\phi)/\sin(M\phi)$ decreases, and as $\phi \to \pi/(M + 1)$ from the left,
\begin{equation}
  \zeta(\phi) = \frac{\pi - (M + 1)\phi}{\sin(\pi M/(M + 1))} e^{\frac{i\pi}{M + 1}} \left( 1 + \bigO \left( \frac{\pi}{M + 1} - \phi \right) \right).
\end{equation}
Thus, for small $\epsilon > 0$, the part of $\tilde{\Sigma}_{\pm}$ in the disc $D_{\epsilon}(0)$ is approximated by the line segments $\{ z = re^{\pm \pi i/(M + 1)} \mid r \leq \epsilon \}$. A plot of $\tilde{\Sigma}$ is shown in the left picture of Figure \ref{fig:SigmaTilde}. Our basic idea is to construct $\Sigma$ by $n\tilde{\Sigma}$. But the contour $\tilde{\Sigma}$ passes through the origin, which coincides with the poles of integrand in \eqref{eq:deformed_contour}, we need to make a small deformation of $\tilde{\Sigma}$ around the origin, which gives the following definition of $\tilde{\Sigma}^{\epsilon}$:
\begin{equation}
  \tilde{\Sigma}^{\epsilon} := \text{$\{ z \in \tilde{\Sigma} \mid \lvert z \rvert \geq \epsilon \}$} \cup \text{the arc of $\{ \lvert z \rvert = \epsilon \}$ connecting $\tilde{\Sigma} \cap \{ \lvert z \rvert = \epsilon \}$ and through $-\epsilon$},
\end{equation}
with counterclockwise orientation. It is clear that $\tilde{\Sigma}^{\epsilon}$ is a closed contour enclosing the interval $[0, 1]$; see the right picture of Figure \ref{fig:SigmaTilde} for an illustration.
\begin{figure}[ht]
  \begin{center}
    \resizebox*{7.5cm}{!}{\includegraphics{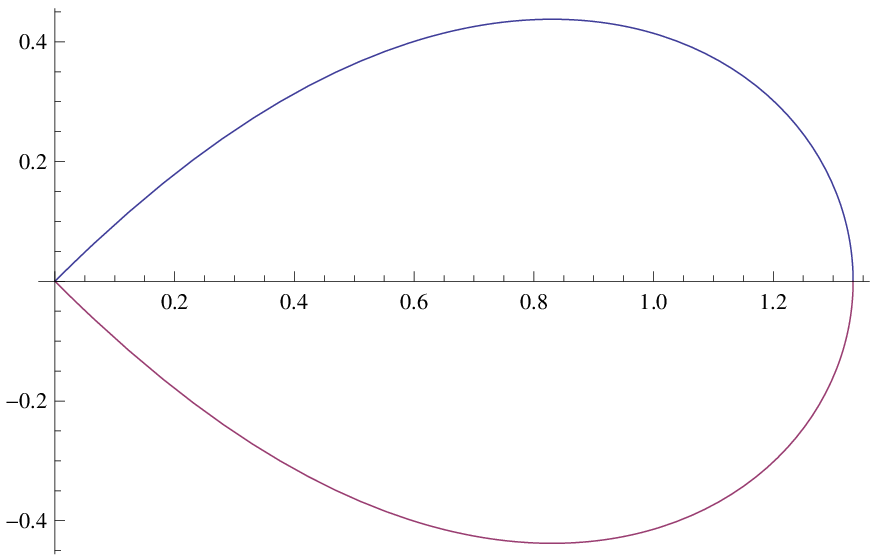}}
    \hspace{2mm}
    \resizebox*{7.5cm}{!}{\includegraphics{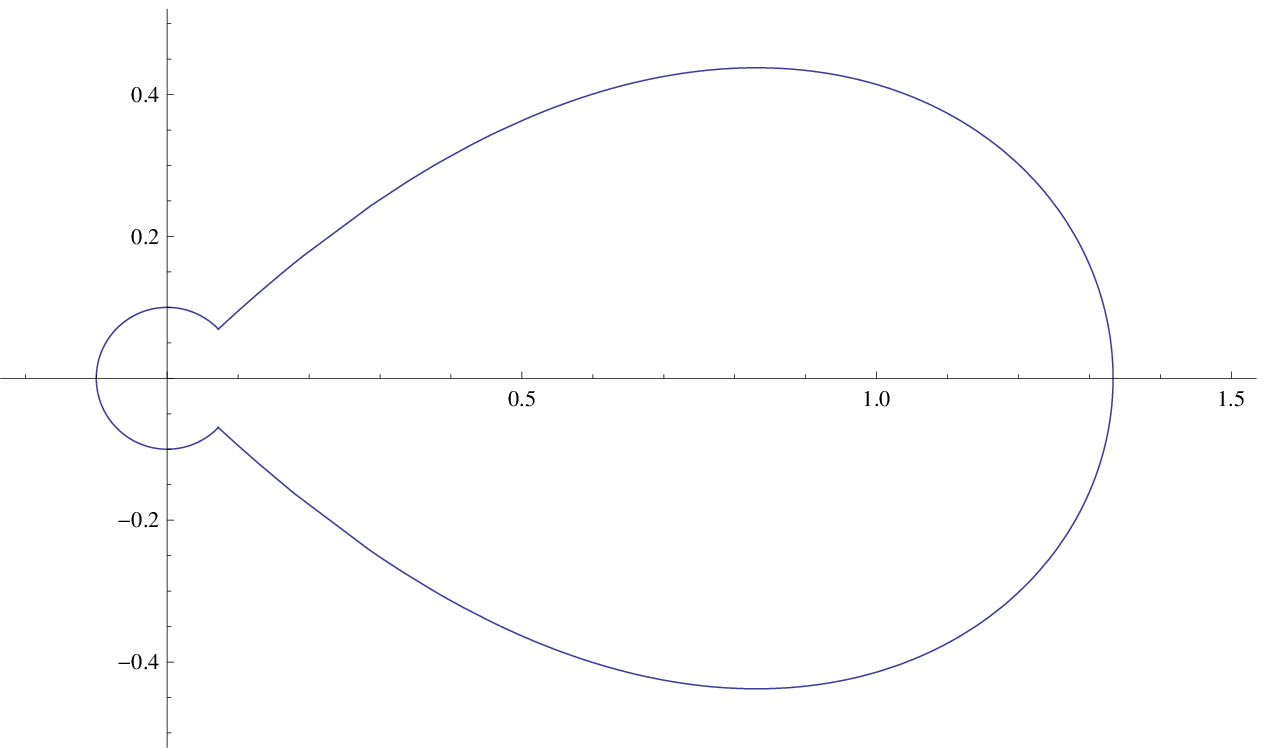}}
    \caption{The contours $\tilde{\Sigma}$ (left picture) and $\tilde{\Sigma}^{\epsilon}$ (right picture) with $M=3$ and $\epsilon=0.1$. }
\label{fig:SigmaTilde}
  \end{center}
\end{figure}

The next two lemmas give the behaviors of $\Re \hat{F}(z;a)$ (defined in \eqref{eq:hat_F}) on the contours $\tilde{\Sigma}$ and $\tilde{\mathcal{C}}_{x_0}$, which will be essential in our later proofs of Lemmas \ref{lem:bulk_ineq} and \ref{lem:edge_ineq}.
\begin{lem} \label{prop:bulk_contours_properties}
  For all $x_0 \in (0, x_*)$, which can be parameterized by $\varphi \in (0, \pi/(M + 1))$ as in \eqref{eq:para x}, there exist constants $\epsilon, \delta > 0$ such that
  \begin{equation} \label{eq:minimum_Sigma_local}
    \Re \hat{F}(z; x_0) \geq \Re \hat{F}(w_{\pm}; x_0) + \epsilon \left\lvert z - w_{\pm} \right\rvert^2 ~~ \text{for $z \in \tilde{\Sigma}^{\epsilon} \cap D_{\delta}(w_\pm)$.}
  \end{equation}
 Moreover, we have
  \begin{equation} \label{eq:minimum_Sigma}
  \begin{aligned}
    \frac{\ud}{\ud\phi} \Re \hat{F}(\zeta(\phi); x_0)
    &\begin{cases}
      < 0 & \text{for $\phi \in (0, \varphi)$}, \\
      > 0 & \text{for $\phi \in \left(\varphi, \frac{\pi}{M + 1}\right)$},
    \end{cases}
   \\
    \frac{\ud}{\ud\phi} \Re \hat{F}(\overline{\zeta(\phi)}; x_0)
    &\begin{cases}
      < 0 & \text{for $\phi \in (0, \varphi)$}, \\
      > 0 & \text{for $\phi \in \left(\varphi, \frac{\pi}{M + 1}\right)$}.
    \end{cases}
    \end{aligned}
  \end{equation}
We also have
\begin{equation} \label{eq:hat_F_along_Sigma_Airy}
  \frac{\ud}{\ud\phi} \Re \hat{F}(\zeta(\phi); x_*) > 0 \quad \text{and} \quad \frac{\ud}{\ud\phi} \Re \hat{F}(\overline{\zeta(\phi)}; x_*) > 0 \quad \text{for $\phi \in \left(0, \frac{\pi}{M + 1}\right)$}.
\end{equation}
\end{lem}

\begin{proof}
  Due to the symmetry of $\Re \hat{F}(z; a)$ with respect to the real axis, it suffices to consider the case that $z \in \compC_+$, that is, only the inequalities involving in $\zeta(\phi)$.

  To show \eqref{eq:minimum_Sigma_local} and \eqref{eq:minimum_Sigma}, we define
  \begin{equation}\label{def:vphi}
    v(\phi) = \frac{(\sin((M+1)\phi))^{M+1}}{\sin\phi(\sin(M\phi))^M}.
  \end{equation}
  Note that $x_0 = v(\varphi)$ and for all $\phi \in [0, \pi/(M + 1)]$,
  \begin{equation}
    \zeta(\phi)^{M + 1} - v(\phi)(\zeta(\phi) - 1) = 0,
  \end{equation}
  where $\zeta(\phi)$ is given in \eqref{eq:parametrisation_of_zeta}. Thus,
  \begin{equation}
    \frac{\ud \hat{F}(\zeta(\phi); x_0)}{\ud \phi} = \frac{\ud \hat{F}}{\ud \zeta} \frac{\ud \zeta(\phi)}{\ud \phi} = \log\left(\frac{\zeta(\phi)^{M+1}}{(\zeta(\phi) - 1)x_0}\right)\frac{\ud \zeta(\phi)}{\ud \phi} = \log\left(\frac{v(\phi)}{v(\varphi)}\right) \frac{\ud \zeta(\phi)}{\ud \phi}.
  \end{equation}
  Since the function $\sin \theta/\sin(c\theta)$ is a strictly decreasing function on $(0,\pi)$ for $0<c<1$, it is readily seen from \eqref{eq:parametrisation_of_zeta} and \eqref{def:vphi} that
  \begin{equation}
  v(\phi) > 0, \quad  \frac{\ud v(\phi)}{\ud \phi} < 0, \quad \Re \frac{\ud \zeta(\phi)}{\ud \phi}=\frac{\ud}{\ud \phi}\left(\frac{\sin((M + 1)\phi)\cos \phi}{\sin(M\phi)}\right) < 0,
  \end{equation}
  for all $\phi \in \left(0, \frac{\pi}{M + 1}\right)$. Hence
\begin{equation} \label{eq:hat_F_Sigma_bulk_first_derivative}
  \frac{\ud \Re \hat{F}(\zeta(\phi); x_0)}{\ud \phi} = \log\left(\frac{v(\phi)}{v(\varphi)}\right) \frac{\ud \Re \zeta(\phi)}{\ud \phi}
  \begin{cases}
    < 0 & \text{if $\phi \in (0, \varphi)$}, \\
    > 0 & \text{if $\phi \in \left(\varphi, \frac{\pi}{M + 1}\right)$},
  \end{cases}
\end{equation}
and
\begin{equation} \label{eq:hat_F_Sigma_bulk_second_derivative}
  \left. \frac{\ud^2 \Re \hat{F}(\zeta(\phi); x_0)}{\ud \phi^2} \right\rvert_{\phi = \varphi} = \left. \frac{\ud}{\ud \phi} \left( \log\left(\frac{v(\phi)}{v(\varphi)}\right) \frac{\ud \Re \zeta(\phi)}{\ud \phi} \right) \right\rvert_{\phi = \varphi} = \left. \frac{\ud \Re \zeta(\phi)}{\ud \phi}\right\rvert_{\phi = \varphi}\frac{v'(\varphi)}{v(\varphi)}  > 0,
\end{equation}
which gives us \eqref{eq:minimum_Sigma_local} and \eqref{eq:minimum_Sigma} for $z$ (or $\zeta(\phi)$)$\in \compC_+$.

Finally, note that the inequality \eqref{eq:hat_F_along_Sigma_Airy} is the limiting case of \eqref{eq:minimum_Sigma} as $x_0 \to x_*$, or equivalently, $\varphi \to 0$, the result is then immediate.
\end{proof}

\begin{lem} \label{lem:for_C}
  For all conjugate pairs $w_\pm \in \compC_\pm$ locating on $\tilde{\Sigma}$, there exists constants $\epsilon, \delta > 0$ such that for all $a \in \realR$
  \begin{equation} \label{eq:maximum_C_local}
    \Re \hat{F}(\Re w_{\pm} + iy; a) \leq \Re \hat{F}(w_{\pm}; a) - \epsilon \left\lvert y - \Im w_{\pm} \right\rvert^2 \quad \text{for $\lvert y - \Im w_{\pm} \rvert \leq \delta$}.
  \end{equation}
  Moreover, we have
  \begin{equation}
  \begin{aligned}
  \label{eq:maximum_C}
    \frac{\ud}{\ud y} \Re \hat{F}( \Re w_{\pm} + iy; a)
    \begin{cases}
      < 0 & \text{if $y > \Im w_+$}, \\
      > 0 & \text{if $y \in \left(0, \Im w_+\right)$}, \\
      < 0 & \text{if $y \in (\Im w_-, 0)$}, \\
      > 0 & \text{if $y < \Im w_-$},
    \end{cases}
    \\
    \lim_{y \to \pm \infty} \frac{\ud}{\ud y} \Re \hat{F}( \Re w_{\pm} + iy; a) = \mp \infty.
    \end{aligned}
  \end{equation}
  We also have, for all $c > 0$,
  \begin{equation} \label{eq:hat_F_along_C_airy}
  \begin{aligned}
    \frac{\ud}{\ud y} \Re \hat{F}(1 + M^{-1} + c + iy; a)
     \begin{cases}
      < 0 & \text{if $y > 0$}, \\
      > 0 & \text{if $y < 0$}, \\
    \end{cases}
    \\
    \lim_{y \to \pm\infty} \frac{\ud}{\ud y} \Re \hat{F}(1 + M^{-1} + c + iy; a)  = \mp \infty.
    \end{aligned}
  \end{equation}
\end{lem}
\begin{proof}
  Similar to the proof of Lemma \ref{prop:bulk_contours_properties}, we need only to prove
   \eqref{eq:maximum_C_local}--\eqref{eq:hat_F_along_C_airy} for $y > 0$.

   By Cauchy-Riemann equations, it follows that, for any $x \in \realR$ and $y \in \realR_+$,
  \begin{equation} \label{primerelation}
    \begin{aligned}
    \frac{\partial}{\partial y}\Re \hat{F}(x + iy; a) &= -\Im \left. \frac{\ud}{\ud z}\hat{F}(z; a) \right\rvert_{z = x + iy}, \\
     \frac{\partial^{2}}{\partial y^{2}}\Re \hat{F}(x + iy; a) & = -\Re \left. \frac{\ud^2}{\ud z^2} \hat{F}(z; a) \right\rvert_{z = x + iy}.
    \end{aligned}
  \end{equation}
  Since
  \begin{equation}
    \hat{F}''(z;a)= \frac{M+1}{z}  -  \frac{1}{z-1},
  \end{equation}
  we have
  \begin{equation}  \label{2prime}
    \frac{\partial^{2}}{\partial y^{2}}\Re \hat{F}(x + iy;a) = \frac{(M+1-Mx)x(x-1)-(Mx+1)y^{2}}{(x^{2}+y^{2})((x-1)^{2}+y^{2})},
  \end{equation}
which is independent of the parameter $a$.

  To show \eqref{eq:maximum_C_local} and \eqref{eq:maximum_C}, we observe from \eqref{eq:defn_w_pm} that
  \begin{equation}\label{eq:range of Rew}
  \Re w_{\pm}=\frac{\sin((M + 1)\varphi)\cos \varphi}{\sin(M\varphi)}\in \left(0,\frac{M+1}{M}\right),\qquad \varphi \in \left(0, \frac{\pi}{M + 1}\right).\end{equation}
  In the case that $0 < \Re w_{\pm} \leq 1$, we have $M+1-M\Re w_{\pm} > 0$. This, together with \eqref{2prime}, implies that
  \begin{equation}\label{eq:second_deri_y}
    \frac{\partial^{2}}{\partial y^{2}}\Re \hat{F}(\Re w_{\pm} + iy; a) < 0 \quad \text{for all $y > 0$}
  \end{equation}
  and
  \begin{equation} \label{eq:limit_sedond_wrt_y}
    \lim_{y \to +\infty} \frac{\partial^{2}}{\partial y^{2}}\Re \hat{F}(\Re w_{\pm} +iy; a) = 0.
  \end{equation}
  Furthermore, since the parameter $a$ is assumed to be real, the value of $\frac{\partial}{\partial y} \Re \hat{F}(x + iy; a)$ does not depend on $a$. By \eqref{primerelation} and \eqref{eq:saddle bulk}, we have
  \begin{equation} \label{eq:vanishing_first_deriv}
    \left. \frac{\partial}{\partial y}\Re \hat{F}(\Re w_{\pm} + iy; a) \right\rvert_{y = \Im w_+} = -\Im \left. \frac{\ud}{\ud z}\hat{F}(z; a) \right\rvert_{z = w_+} = -\Im \left. \frac{\ud}{\ud z}\hat{F}(z; x_0) \right\rvert_{z = w_+} = 0,
  \end{equation}
  for any $a\in\mathbb{R}$. Thus, $\Re \hat{F}(\Re w_{\pm} + y; a)$, as a function of $y>0$, has a critical point at $\Im w_+$, and by \eqref{eq:second_deri_y}, is a concave function attaining its maximum at $\Im w_+$. We thus prove \eqref{eq:maximum_C_local} and \eqref{eq:maximum_C} in this case.

 In the case that $1< \Re w_{\pm} <(M+ 1)/M$, the equation $\frac{\partial^{2}}{\partial y^{2}}\Re \hat{F}(\Re w_{\pm} + iy;a)=0$ has a unique real root at
 \begin{equation}
   y^* = \sqrt{\frac{(M+1-M \Re w_{\pm}) \Re w_{\pm} (\Re w_{\pm} - 1)}{M\Re w_{\pm} + 1}},
 \end{equation}
for $y \in [0, \infty)$. Thus, $\frac{\partial}{\partial y}\Re \hat{F}(\Re w_{\pm} + iy; a)$ is strictly increasing if $y \in [0, y^*)$, and strictly decreasing if $y \in (y^*, \infty)$. Note that $\frac{\partial}{\partial y}\Re \hat{F}(\Re w_{\pm} + iy;a)$ is a continuous odd function in $y$, one has $\frac{\partial}{\partial y}\Re \hat{F}(\Re w_{\pm} + iy;a)\Big{|}_{y=0}=0$. Therefore,
\begin{equation}
  \frac{\partial}{\partial y}\Re \hat{F}(\Re w_{\pm} + iy;a) > 0, \qquad y\in(0, y^*).
\end{equation}
On the other hand, by \eqref{eq:vanishing_first_deriv}, we have that $ \frac{\partial}{\partial y}\Re \hat{F}(\Re w_{\pm} + iy; a)$ vanishes at $\Im w_+$. Thus we conclude that  $\Im w_+ \in (y^*, \infty)$, and have that on the interval $[y^*, \infty)$, the function $\Re \hat{F}(\Re w_{\pm} + iy; a)$ has a critical point at $\Im w_+$, and is a concave function with the maximum at $\Im w_+$. Note that \eqref{eq:limit_sedond_wrt_y} also holds in this case. We thus prove \eqref{eq:maximum_C_local} and \eqref{eq:maximum_C} in this case.

 We finally prove \eqref{eq:hat_F_along_C_airy}. By substituting $x = 1 + M^{-1} + c$ into \eqref{2prime}, we have
 \begin{equation}
   \frac{\partial^{2}}{\partial y^{2}}\Re \hat{F}(1 + M^{-1} + c + iy; a) < 0 \qquad \text{for all $y > 0$}.
 \end{equation}
 On the other hand, since $\frac{\partial}{\partial y}\Re \hat{F}(1 + M^{-1} + c + iy;a)$ is a continuous odd function in $y$, its value at $0$ is $0$. We conclude that $\frac{\partial}{\partial y}\Re \hat{F}(1 + M^{-1} + c + iy;a) < 0$ for all $y > 0$, and this gives us \eqref{eq:hat_F_along_C_airy}.
\end{proof}

In Sections \ref{subsec:proof_of_bulk_univ} and \ref{subsec:proof_edge}, the contours $\mathcal{C}$ and $\Sigma$ in \eqref{eq:deformed_contour} are constructed from $\tilde{\mathcal{C}}_{x_0}$ and $\tilde{\Sigma}^r$, where $r$ depends on $n$ and a small parameter. In the proofs of our main theorems, we need to estimate some integrals over specified contours, which relies on Lemmas \ref{lem:bulk_ineq} and \ref{lem:edge_ineq} concerning the inequalities satisfied by $\Re F$ over $\mathcal{C}$ and $\Sigma$. We are now ready to prove these two lemmas based on Lemmas \ref{prop:bulk_contours_properties} and \ref{lem:for_C}.

\subsection{Proof of Lemma \ref{lem:bulk_ineq}} \label{subsec:proof_lem_bulk_ineq}

For notational convenience, we shall write $F(z; n^M x_0)$ as $F(z)$ throughout this subsection.

\paragraph{Proof of \eqref{eq:minimum_Sigma_local_F} and \eqref{eq:minimum_Sigma_F}}
Recall the contour $\Sigma_{\curved}$ defined by \eqref{def:sigm_curved_bulk} and \eqref{eq:Sigma_contour_bulk_2}, we further write it as
\begin{equation}
  \Sigma_{\curved} = \Sigma_{\lef} \cup \Sigma_{\rig},
\end{equation}
where
\begin{equation}
  \begin{aligned}
    \Sigma_{\lef} := \{ z \in \Sigma_{\curved} \mid \lvert z \rvert = nr \}, \qquad
    \Sigma_{\rig} := \{ z \in \Sigma_{\curved} \mid \lvert z \rvert > nr \},
  \end{aligned}
\end{equation}
\ie, $\Sigma_{\lef}$ is part of a circle centring at $0$ with radius $nr$, and $\Sigma_{\rig}$ is the part of $\Sigma_{\curved}$ that does not overlap the circle.

If $z \in \Sigma_{\rig}$, it can be expressed as $z = n \zeta(\phi)$ or $z = n \overline{\zeta(\phi)}$ for some $\phi \in (0, \pi/(M + 1))$ by \eqref{eq:parametrisation_of_zeta}, so there exists a constant $\varepsilon' > 0$, such that for large enough $n$, $\arg (z + \nu_j + 1) \in (-\pi + \varepsilon', \pi - \varepsilon')$ and $\arg (z - n + 1) \in (-\pi + \varepsilon', \pi - \varepsilon')$. We then apply the Stirling's formula \eqref{eq:stirling} to $\Gamma(z + \nu_j + 1)$ and $\Gamma(z - n + 1)$ in formula \eqref{eq:defn_F}, and obtain a uniform approximation of $F(z)$ by $n\hat{F}(z/n;x_0)$, on account of \eqref{eq:stirling}--\eqref{eq:hat_F}. Thus, the inequalities \eqref{eq:minimum_Sigma_local} and \eqref{eq:minimum_Sigma} for $\Re \hat{F}(z; x_0)$ on $\tilde{\Sigma}$ yield the desired inequalities \eqref{eq:minimum_Sigma_local_F} and \eqref{eq:minimum_Sigma_F} for $z \in \Sigma_{\rig}$.

If $z \in \Sigma_{\lef}$, Stirling's formula \eqref{eq:stirling} may not be valid anymore, and we need to pay special attention. Note that there exists a constant $\varepsilon' > 0$, such that for all $n$ large enough, $\arg (-z - \nu_j) \in (-\pi + \varepsilon', \pi - \varepsilon')$ and $\arg (n - z) \in (-\pi + \varepsilon', \pi - \varepsilon')$. We make use of  the reflection formula of gamma function
\begin{equation} \label{eq:reflection_gamma}
  \Gamma(z) \Gamma(1 - z) = \frac{\pi}{\sin(\pi z)}
\end{equation}
to obtain a uniform approximation of $F(z)$. Since
\begin{equation}
 \sin(\pi(z - n + 1)) = \pm \sin(\pi z),\qquad \sin(\pi(z + \nu_j + 1)) = \pm \sin(\pi z),
\end{equation}
we have
\begin{equation} \label{eq:Re_F_left}
  \begin{split}
    \Re F(z) = {}& \Re \log \left( \frac{ \Gamma(n - z)}{\prod_{j=0}^M \Gamma(-z - \nu_j)} \frac{\pi^M \sin(\pi(z - n + 1))}{\prod^M_{j = 0} \sin(\pi(z + \nu_j + 1))} \left(n^M x_0\right)^{-z} \right) \\
    = {}& \Re \log \left( \frac{ \Gamma(n - z)}{\prod_{j=0}^M \Gamma(-z - \nu_j)} \left(n^M x_0\right)^{-z} \right) - M \log \frac{\lvert \sin(\pi z) \rvert}{\pi} \\
    = {}& \Re \tilde{G}\left(z; n^M x_0\right) - M \log \lvert 2\sin(\pi z) \rvert \\
    & + \sum_{j=0}^M\left(\nu_j+\frac{1}{2}\right) \log |z| - \frac{1}{2} \log|z - n| + \frac{M}{2} \log(2\pi) +\bigO\left(n^{-1}\right),
  \end{split}
\end{equation}
where
\begin{equation} \label{eq:tilde_G_left}
  \tilde{G}\left(z; n^M x_0\right) = (M + 1)z(\log(-z) - 1) - (z - n)(\log(n - z) - 1) - (M \log n + \log x_0) z.
\end{equation}
It is also straightforward to check that
\begin{equation} \label{eq:tilde_G_in_hat_G}
  \tilde{G}\left(z; n^M x_0\right) = n \hat{G}\left(\frac{z}{n}; x_0\right) + n\log n,
\end{equation}
where
\begin{equation} \label{eq:hat_G_defn}
  \hat{G}(\zeta; x_0) = (M + 1)\zeta(\log(-\zeta) - 1) - (\zeta - 1)(\log(1 - \zeta) - 1) - \zeta \log x_0.
\end{equation}

Formulas \eqref{eq:Re_F_left}--\eqref{eq:hat_G_defn} constitute a uniform approximation of $\Re F(z)$ for $z \in \Sigma_{\lef}$. Now we choose the parameter $\epsilon'$ in \eqref{eq:Sigma_contour_bulk_2} small enough such that $nw_\pm \in \Sigma_{\rig}$, thus $\Re F(nw_{\pm})$ can be approximated by \eqref{eq:F_in_F_tilde}--\eqref{eq:hat_F}. The
inequalities \eqref{eq:minimum_Sigma_local_F} and \eqref{eq:minimum_Sigma_F}
follow if we can show that there exists a constant $c > 0$ such that for all large enough $n$ and $z \in \Sigma_{\lef}$,
\begin{equation} \label{eq:Sigma_left_ineq_bulk}
  \Re \hat{G}\left(\frac{z}{n}; x_0\right) - \frac{M}{n} \log \lvert 2\sin(\pi z) \rvert > \Re \hat{F}(w_{\pm};x_0) + c.
\end{equation}
To prove \eqref{eq:Sigma_left_ineq_bulk}, we note that, by \eqref{eq:minimum_Sigma} in Lemma \ref{prop:bulk_contours_properties},
\begin{equation}
 \Re \hat{G}(0;x_0) = \Re \hat{F}(0;x_0) > \Re \hat{F}(w_{\pm};x_0),
\end{equation}
so we simply take
\begin{equation}
  c = \frac{1}{3}\left(\Re \hat{F}(w_{\pm};x_0) - \Re \hat{F}(0;x_0)\right).
\end{equation}
Since $\Re \hat{G}(\zeta;x_0)$ is a continuous function in the vicinity of $0$, we have that if $\epsilon$ is small enough, or equivalently, $r$ is small enough, $\lvert \Re \hat{G}(z/n;x_0) - \Re \hat{G}(0;x_0) \rvert < c$ for all $z \in \Sigma_{\lef}$. On the other hand, it is straightforward to check that if $\epsilon'$ is small enough, and $n$ is large enough, then $M n^{-1} \log \lvert 2\sin(\pi z) \rvert < c$ for all $z \in \Sigma_{\lef}$. Thus \eqref{eq:Sigma_left_ineq_bulk} holds if $\epsilon'$ is small enough while $n$ is large enough. This completes the proof of \eqref{eq:minimum_Sigma_local_F} and \eqref{eq:minimum_Sigma_F}.

\paragraph{Proof of \eqref{eq:maximum_C_local_F}--\eqref{eq:maximum_C_global_F}}

For any $x_0 \in (0, (M + 1)^{M + 1}/M^M)$, the associated complex conjugate numbers $w_\pm$ satisfying $\Re w_{\pm} \in (0, 1 + M^{-1})$; see \eqref{eq:range of Rew}. We prove the inequalities in three cases depending on the value of $\Re w_{\pm}$.

We first consider the case that $\Re w_{\pm} > 1$, or equivalently, the vertical contour $\mathcal{C}$ defined in \eqref{eq:C_contour_bulk} is on the right of $n$. Then for all $z \in \mathcal{C}$, there exists a constant $\varepsilon' > 0$ such that $\arg (z + \nu_j + 1) \in (-\pi + \varepsilon', \pi - \varepsilon')$ and $\arg (z - n + 1) \in (-\pi + \varepsilon', \pi - \varepsilon')$ for large $n$ enough. The formulas \eqref{eq:F_in_F_tilde}--\eqref{eq:hat_F} then give a uniform approximation of $F(z)$ by $n\hat{F}(z/n;x_0)$, similar to the case that $z \in \Sigma_{\rig}$ discussed previously. Hence, \eqref{eq:maximum_C_local_F}--\eqref{eq:maximum_C_global_F} are direct consequence of \eqref{eq:maximum_C_local} and \eqref{eq:maximum_C} in Lemma \ref{lem:for_C}.

In the case that $\Re w_{\pm} \in (0, 1)$, or equivalently, the vertical contour $\mathcal{C}$ lies between $0$ and $n$, we divide
\begin{equation} \label{eq:outer_inner_C}
  \mathcal{C} = \mathcal{C}_{\out} \cup \mathcal{C}_{\inner},
\end{equation}
where
\begin{equation}
 \mathcal{C}_{\out} = \{ z \in \mathcal{C} \mid \lvert \Im z \rvert > n \epsilon' \}, \quad \mathcal{C}_{\inner} = \{ z \in \mathcal{C} \mid \lvert \Im z \rvert \leq n\epsilon' \}
\end{equation}
and $\epsilon'$ is a small positive number.

For $z \in \mathcal{C}_{\out}$, we can still use the Stirling's formula directly and approximate $F(z)$ by $n\hat{F}(z/n;x_0)$ through \eqref{eq:F_in_F_tilde}--\eqref{eq:hat_F}. The desired inequalities \eqref{eq:maximum_C_local_F}--\eqref{eq:maximum_C_global_F} for such $z$ again follow from \eqref{eq:maximum_C_local} and \eqref{eq:maximum_C} in Lemma \ref{lem:for_C}.

For $z \in \mathcal{C}_{\inner}$, we encounter the problem of validity of Stirling's formula for $\Gamma(z - n + 1)$. With the aid of the reflection formula \eqref{eq:reflection_gamma}, for $n$ large enough, we obtain the following uniform approximation of $\Re F(z)$ given by
\begin{multline}
  \Re F(z) = \Re \tilde{H}(z; n^M x_0) + \log \lvert 2\sin(\pi z) \rvert \\
  + \sum_{j=0}^M\left(\nu_j+\frac{1}{2}\right) \log |z| - \frac{1}{2} \log|z - n| + \frac{M}{2} \log(2\pi) +\bigO\left(n^{-1}\right),
\end{multline}
where
\begin{equation}
  \begin{split}
    \tilde{H}\left(z; n^M x_0\right) = {}& (M + 1)z(\log z - 1) - (z - n)(\log(n - z) - 1) - (M \log n + \log x_0) z \\
    = {}& n \hat{H}\left(\frac{z}{n}; x_0 \right) + n\log n,
  \end{split}
\end{equation}
and
\begin{equation}
  \hat{H}(\zeta; x_0) = (M + 1)\zeta(\log \zeta - 1) - (\zeta - 1)(\log(1 - \zeta) - 1) - \zeta \log x_0.
\end{equation}
Similar to the discussions used in the proof of \eqref{eq:minimum_Sigma_local_F} and \eqref{eq:minimum_Sigma_F} with $z \in \Sigma_{\lef}$, we only need to show that for $z \in \mathcal{C}_{\inner}$, there exists a constant $c > 0$ such that
\begin{equation}
  \Re \hat{H}\left(\frac{z}{n}; x_0\right) + \frac{1}{n} \log \lvert 2\sin(\pi z) \rvert < \Re \hat{F}(w_{\pm};x_0) - c.
\end{equation}
Now we take
\begin{equation}
  c = \frac{1}{3}\left(\Re \hat{F}(w_{\pm};x_0) - \Re \hat{F}(\Re w_{\pm};x_0)\right),
\end{equation}
which is positive by \eqref{eq:maximum_C} in Lemma \ref{prop:bulk_contours_properties}. Since $\Re \hat{H}(\zeta;x_0)$ is continuous in the vicinity of $\Re w_{\pm}$, we have that $\lvert \Re \hat{H}(z/n;x_0) - \Re \hat{H}(\Re w_{\pm};x_0) \rvert < c$ for all $z \in \mathcal{C}_{\inner}$ if $\epsilon'$ is small enough. On the other hand, it is straightforward to check that if $\epsilon'$ is small enough and $n$ large enough, then $n^{-1} \log \lvert 2 \sin(\pi z) \rvert < c$ for all $z \in \mathcal{C}_{\inner}$. This gives us \eqref{eq:maximum_C_F} for $z \in \mathcal{C}_{\inner}$, and finishes the proof in this case.

Finally, if $\Re w_{\pm} = 1$, we still divide $\mathcal{C}$ into $\mathcal{C}_{\out}$ and $\mathcal{C}_{\inner}$ as in \eqref{eq:outer_inner_C}. The estimate of $\Re F(z)$ on $\mathcal{C}_{\out}$ can be derived from the Stirling's formula, but for $z \in \mathcal{C}_{\inner}$, we need to control the value of $\Gamma(z - n + 1)$ for $z - n = o(n)$. Since the strategy is similar, we omit the details here.

\subsection{Proof of Lemma \ref{lem:edge_ineq}} \label{subsec:proof_lem_edge_ineq}
For notational convenience, we shall write $F(z; n^M x_*)$ as $F(z)$ throughout this subsection.
\paragraph{Proof of \eqref{eq:F_on_Sigma_edge}}

For $z \in \Sigma_{\curved}$, the proof is parallel to that of \eqref{eq:minimum_Sigma_F}. The only difference is that after approximating $F(z)$ uniformly by $n\hat{F}(z/n; x_*)$ (defined in \eqref{eq:hat_F}) or by $n\hat{G}(z/n; x_*)$ (defined in \eqref{eq:hat_G_defn}, with $x_0$ replaced by $x_*$), depending on whether $\lvert z \rvert > nr$ or $\lvert z \rvert = nr$, we compare $\hat{F}(z/n; x_*)$ and $\hat{G}(z/n; x_*)$ with $\Re \hat{F}(1 + M^{-1}; x_*)$, instead of $\Re \hat{F}(w_{\pm};x_*)$ used in the proof of \eqref{eq:minimum_Sigma_F}. We then apply the inequality \eqref{eq:hat_F_along_Sigma_Airy}, instead of the inequality \eqref{eq:minimum_Sigma}, in the comparison. The details are left to the interested readers.

For $z \in \Sigma_{\vertical}$, we apply the approximation of $F(z)$ by $n\hat{F}(z/n; x_*)$ as in \eqref{eq:stirling}--\eqref{eq:hat_F}, and reduce the proof of \eqref{eq:F_on_Sigma_edge} for $z \in \Sigma_{\vertical}$ to proving
\begin{equation} \label{eq:ineq_Sigma_vertical}
  \Re \hat{F}\left(\frac{z}{n}; x_*\right) > \Re \hat{F}\left(1 + M^{-1}; x_*\right) + \delta n^{-\frac{9}{10}},\quad \delta>0.
\end{equation}
Without loss of generality, we show \eqref{eq:ineq_Sigma_vertical} only for $z \in \Sigma_{\vertical} \cap \compC_+$. By \eqref{eq:maximum_C} in Lemma \ref{lem:for_C}, $\Re \hat{F}(z/n; x_*)$ increases as $\Im z$ increases for $z \in \Sigma_{\vertical} \cap \compC_+$. So we only need to check that \eqref{eq:ineq_Sigma_vertical} holds for $z = nz_0 + c_1 n^{\frac{7}{10}} e^{2\pi i/3}$, i.e., the lowest end of $\Sigma_{\vertical} \cap \compC_+$. The explicit computation in \eqref{eq:F_hat_Taylor_expansion} gives the approximation of $\Re \hat{F}(z/n; x_*)$ at this point and finishes the proof in this case.

\paragraph{Proof of \eqref{eq:F_on_C_edge} and \eqref{eq:F_on_C_infty_edge}}

For all $z \in \mathcal{C}_{\glob}$, the uniform approximation of $F(z)$ by $n\hat{F}(z/n; x_*)$ as in \eqref{eq:stirling}--\eqref{eq:hat_F} is valid. Then we reduce \eqref{eq:F_on_C_edge} and \eqref{eq:F_on_C_infty_edge} to
\begin{align}
  \Re \hat{F}(z/n; x_*) < {}& \Re \hat{F}(z_0;x_*) - \delta n^{-\frac{9}{10}} & & \text{for $z \in \mathcal{C}_{\glob}$}, \label{eq:F_on_C_edge_approx} \\
  \Re \hat{F}(z/n; x_*) < {}& \Re \hat{F}(z_0;x_*) - \delta \lvert z \rvert/n & & \text{for $z \in \mathcal{C}_{\glob}$ and $\lvert z \rvert > \delta^{-1} n$}. \label{eq:F_on_C_infty_edge_approx}
\end{align}
The inequality \eqref{eq:F_on_C_infty_edge_approx} is a direct consequence of \eqref{eq:hat_F_along_C_airy}. To prove \eqref{eq:F_on_C_edge_approx} for $z\in \mathcal{C}_{\glob} \cap \compC_+$, we note that $\Re \hat{F}(z/n; x_*)$ decreases as $\Im z$ increases, as shown in \eqref{eq:hat_F_along_C_airy}. Thus we only need to check \eqref{eq:F_on_C_edge_approx} at $z = nz_0 + c_1 n^{\frac{7}{10}} e^{\pi i/3}$, the lowest end of $\mathcal{C}_{\glob} \cap \compC_+$. The explicit computation \eqref{eq:F_hat_Taylor_expansion} gives the approximation of $\Re \hat{F}(z/n; x_*)$ at this point and finishes the proof in this case. The inequality \eqref{eq:F_on_C_edge_approx} for $z\in\mathcal{C}_{\glob} \cap \compC_-$ can be proved in the same way.

\section{Bulk and soft edge universality in other product models}\label{sec:extensions}

As mentioned in Section \ref{subsec:aboutproof}, our strategy presented before is not restricted to the particular model, but applicable to other interesting models of products of random matrices. In this section, we demonstrate this aspect by establishing bulk and soft edge universality in two more examples with sketched proofs. One example is the product of $M$ Ginibre matrices and the inverse of $K$ Ginibre matrices studied by Forrester \cite{Forrester14}, and the other example is the product of $M-1$ Ginibre matrices with one truncated unitary matrix considered by Kuijlaars and Stivigny \cite{Kuijlaars-Stivigny14}. Our method can be applied to more cases, notably the newly analysed model in \cite{Forrester-Liu15} and models that can be expressed in the general double contour integral formalism in \cite{Claeys-Kuijlaars-Wang15}.

In this section, we use the same notations as in previous sections for objects in different models that have counterpart in the model introduced and computed in Sections \ref{sec:Intro}--\ref{sec:ContConst}. We hope these notations show the readers analogue in our arguments while do not bring confusion.

\subsection{Products of Ginibre matrices and their inverses}

This model refers to the product
\begin{equation} \label{Ymk}
Y_{M,K} = X_M   \cdots X_1(\tilde{X}_K   \cdots \tilde{X}_1)^{-1},
\end{equation}
where $X_j$, $j=1,\cdots,M$, and $\tilde{X}_k$, $k=1,\cdots,K$, are complex Ginibre random matrices with size $(n+\nu_j)\times (n+\nu_{j-1})$ and $(n+\tilde{\nu}_k)\times (n+\tilde{\nu}_{k-1})$, respectively. We assume that
\begin{equation}\label{eq:nu_tildnu}
\nu_0=\tilde{\nu}_0=\tilde{\nu}_K=0, \qquad \nu_j, \tilde{\nu}_k\geq 0,
\end{equation}
thus, $Y_{M,K}$ is a rectangular matrix of size $(n+\nu_M)\times n$. Clearly, $Y_{M,K}$ extends products of Ginibre matrices $Y_M$ defined in \eqref{Ym}.

It was shown in \cite[Propostion 5]{Forrester14} that the squared singular values of $Y_{M,K}$ forms a determinantal process with the correlation kernel
\begin{multline} \label{Knintegralinverse}
    K_n(x,y) =  \frac{1}{(2\pi i)^2} \int_{-1/2-i\infty}^{-1/2+i\infty} \ud s \oint_{\Sigma}  \ud t\\
        \prod_{j=0}^M   \frac{\Gamma(s+\nu_j+1)}{\Gamma(t+\nu_j+ 1)}\,  \prod_{k=1}^K   \frac{\Gamma(n-s+\tilde{\nu}_k)}{\Gamma(n-t+\tilde{\nu}_k)}
           \, \frac{\Gamma(t-n+1)}{\Gamma(s-n+1)}
        \frac{x^t y^{-s-1}}{s-t},
\end{multline}
where $\Sigma$ is a closed contour going around $0, 1, \ldots, n-1$ in the positive direction and  $\Re t > -1/2$ for $t \in \Sigma$.

The special case $(M,K)=(M,0)$ is the model we considered in Sections \ref{sec:Intro}--\ref{sec:ContConst}, while the special case $(M,K)=(0,K)$ is equivalent spectrally to the model $(K, 0)$ by reciprocal transform. In the generic case $K > 0, M>0$, the limiting mean density is supported over the whole positive real axis as $n\to\infty$ (see \cite{Forrester14,Forrester-Liu14}), which implies that no soft edge occurs. Below we only consider the $K > 0, M > 0$ case.

To state our result for bulk universality, we need the following parametrization of the spectral parameter $x_0$     \begin{equation}\label{eq:parainverse}
x_0=\frac{
\left(\sin \left( \frac{M+1}{K+1} \varphi+\frac{K  }{K+1}\pi \right)\right)^{M+1} }{\left(\sin\varphi\right)^{K+1} \,\left(\sin \left( \frac{M-K}{K+1} \varphi+\frac{K  }{K+1}\pi \right)\right)^{M-K}}, \quad 0<\varphi<\frac{\pi}{M+1},
\end{equation}
which is a one-to-one mapping from $(0,\pi/(M+1))$ to $(0,+\infty)$; see \cite{Forrester-Liu14,Haagerup-Moller13}.

\begin{thm}[Bulk universality]\label{thm:bulk_inverse}
Let $K_n(x,y)$ be the correlation kernel defined in \eqref{Knintegralinverse}. For $x_0\in\left(0, +\infty \right)$, which is parametrized by $\varphi \in (0, \pi/(M+1))$ through \eqref{eq:parainverse}, we have, with $\nu_j,\tilde{\nu}_k$ being fixed,
\begin{equation} \label{eq:bulk_inverse}
 \lim_{n \to \infty} \frac{e^{-\pi \xi\cot \varphi  }}{e^{-\pi \eta\cot \varphi   }} \frac{n^{M-K-1}}{\rho(\varphi)} K_n\left(n^{M-K}\left( x_0+ \frac{  \xi}{ n\rho(\varphi)}\right), n^{M-K}\left( x_0+ \frac{  \eta}{n\rho(\varphi) }\right)\right)  =K_{\sin}(\xi,\eta)
\end{equation}
uniformly for $\xi$ and $\eta$ in any compact subset of $\mathbb{R}$, where the function $\rho$ is given by
\begin{equation} \label{rhoinverse}
  \rho (\varphi) = \frac{1}{\pi x_0}\frac{\sin \left( \frac{M+1}{K+1} \varphi+\frac{K  }{K+1}\pi \right)  }{  \sin \left( \frac{M-K}{K+1} \varphi+\frac{K  }{K+1}\pi \right) }\,  \sin\varphi.
\end{equation}
\end{thm}

We now give a sketched proof of the above theorem with emphasis on the key steps.

\begin{proof}[Sketched proof of Theorem \ref{thm:bulk_inverse}]
We scale the values of $x$ and $y$ in \eqref{Knintegralinverse} such that
\begin{equation}
x=n^{M-K}\left( x_0+ \frac{  \xi}{n\rho(\varphi) }\right), \qquad y=n^{M-K}\left( x_0+ \frac{  \eta}{n\rho(\varphi) }\right),
\end{equation}
where $\xi, \eta\in\mathbb{R}$ and $\rho(\varphi)$ is given in \eqref{rhoinverse}. By Stirling's formula \eqref{eq:stirling} and the reflection formula, it follows that, for $n$ large,
\begin{align}
  K_n &\left(x,y \right) \sim  -\frac{n^{-M+K}}{(2\pi i)^2}\int_{\mathcal{C}} \ud s \oint_{\Sigma}  \ud t \frac{e^{n(\hat{F}(ns; x_0) - \hat{F}(nt; x_0))}}{s-t} \, \left(1+\frac{\xi}{n x_0\rho }\right)^{t} \left(1+\frac{\eta}{nx_0\rho }\right)^{-s} \nonumber\\
  &\left(x_0+\frac{\eta}{\rho n}\right)^{-1}\, \exp\left\{\sum_{j=0}^M\left(\nu_j+\frac{1}{2}\right)\log\frac{s}{t}+\sum_{k=1}^K\left(\tilde{\nu}_k-\frac{1}{2}\right)\log\frac{n - s}{n - t}-\frac{1}{2}\log\frac{s - n}{t - n} \right\},
\end{align}
where the shapes of the contours $\mathcal{C}$ and $\Sigma$ are to be described later. Here,
\begin{equation}\label{Ffunctioninverse}
  \hat{F}(z; x_0)=(M+1)(z\log z-1)+K(1-z)(\log(1- z)-1)-(z-1)(\log(z-1)-1)-z\log x_0.
\end{equation}
Since
  \begin{equation}
   \hat{F}_z(z; x_0) = (M+1)\log z-K\log(1-z)-\log(z-1)-\log x_0,
   \end{equation}
the saddle point of $\hat{F}(z; x_0)$ satisfies the following algebraic equation
\begin{equation} z^{M+1}+x_0 (1-z)^{K+1}=0.\label{aqinverse}
\end{equation}
Particularly, with the help of parametrization \eqref{eq:parainverse}, two solutions of \eqref{aqinverse} can be given explicitly by
\begin{equation}
  w_{\pm}=\frac{\sin \left( \frac{M+1}{K+1} \varphi+\frac{K  }{K+1}\pi \right) } {  \sin \left( \frac{M-K}{K+1} \varphi+\frac{K }{K+1}\pi \right) }\,  e^{\pm i \varphi};\label{twosinverse}
 \end{equation}
see \cite{Forrester-Liu14} for more details. These two complex conjugate numbers play the same role of $w_{\pm}$ used in Section \ref{subsec:proof_of_bulk_univ}.

Similar to the contours used in Section \ref{subsec:proof_of_bulk_univ}, the contour $\mathcal{C}$ is chosen to be the straight line
\begin{equation}
\mathcal{C}=\left\{z~\Big{|}~\Re \frac{z}{n} =\Re w_{\pm} = \frac{\sin \left( \frac{M+1}{K+1} \varphi+\frac{K  }{K+1}\pi \right)  }{  \sin \left( \frac{M-K}{K+1} \varphi+\frac{K  }{K+1}\pi \right) }\,  \cos\varphi\right\},
\end{equation}
while $\Sigma$ is chosen to be a deformation based on the contour
\begin{equation}
  \tilde{\Sigma}=\left\{z=\frac{\sin \left( \frac{M+1}{K+1} \phi+\frac{K  }{K+1}\pi \right)  }{  \sin \left( \frac{M-K}{K+1} \phi+\frac{K  }{K+1}\pi \right) }\,   e^{i\phi}~\Big{|}~-\frac{\pi}{M+1}\leq \phi\leq \frac{\pi}{M+1}\right\},
\end{equation}
in the same way as the $\Sigma$ in Section \ref{subsec:proof_of_bulk_univ}.

One can then show that, in a manner similar to Lemma \ref{lem:bulk_ineq}, $\Re \hat{F}(z; x_0)$ defined in \eqref{Ffunctioninverse} attains its global maximum at  $z=w_{\pm}$ for $nz \in \mathcal{C}$ and its global minimum at $z = w_{\pm}$ for $z \in \tilde \Sigma$. This in turn implies that the main contribution of the integral in \eqref{Knintegralinverse} comes from the counterpart of the integral $I_2$ in \eqref{eq:formula_I_2}, that is,
\begin{align} \label{eq:Kn_icorr_inverse}
  K_n \left(x,y \right) &\sim   \frac{n^{-M+K}}{ 2\pi i x_0}   \int_{nw_{-}}^{nw_+} \ud s  \, \left(1+\frac{\xi}{x_0\rho n}\right)^s \left(1+\frac{\eta}{x_0\rho n}\right)^{-s}\nonumber \\
  &\sim n^{-M+K+1}\, \rho\, e^{\pi (\xi-\eta)\cot\varphi}\, \frac{\sin\pi (\xi-\eta)}{\pi(\xi-\eta)}.
\end{align}
 \end{proof}

\subsection{Product of Ginibre matrices with one truncated unitary matrix}
This model refers to the product
\begin{equation} \label{Ym1}
Y  = X_M   \cdots X_2 V,
\end{equation}
where $X_i$, $i=2,\ldots,M$ is a Ginibre matrix of size $(n+\nu_i)\times (n+\nu_{i-1})$ with $\nu_i\geq 0$. The $(n+\nu_1)\times n$ matrix $V$ is the left upper block of an $l\times l$ Haar distributed unitary matrix $U$ with $l\geq 2n+\nu_1$. It is known that the squared singular values of $V$ are distributed on $(0,1)$ according to a Jacobi unitary ensemble; cf. \cite{Jiang09}.

By \cite[Proposition 4.4]{Kuijlaars-Stivigny14}, we have that the squared singular values of $Y$ form a determinantal process with the correlation kernel
\begin{multline} \label{Knintegralone}
    K_n(x,y) =  \frac{1}{(2\pi i)^2} \int_{-1/2-i\infty}^{-1/2+i\infty} \ud s \oint_{\Sigma}  \ud t\\
        \prod_{j=0}^M   \frac{\Gamma(s+\nu_j+1)}{\Gamma(t+\nu_j+ 1)} \, \frac{\Gamma(t-n+1)}{\Gamma(s-n+1)}\frac{\Gamma(t+n+\kappa)}{\Gamma(s+n+\kappa)}
        \frac{x^t y^{-s-1}}{s-t},
\end{multline}
where
\begin{equation}
  \nu_0=0, \qquad \kappa:=l+1-2n>\nu_1,
\end{equation}
$\Sigma$ is a closed contour going around $0, 1, \ldots, n-1$ in the positive direction and  $\Re t > -1/2$ for $t \in \Sigma$. To state the universal results for the correlation kernel, we need the following parametrization
\begin{equation}\label{eq:paraunitary}
x_0=\frac{\left(\sin \left( \frac{M+1}{2} \varphi \right)\right)^{\frac{M+1}{2}}}{\sin\varphi \,\left(\sin\left( \frac{M-1}{2}\varphi\right)\right)^{\frac{M-1}{2}}}, \quad 0< \varphi<\frac{2\pi}{M+1},
\end{equation}
which is a one-to-one mapping from  $(0,2\pi/(M+1))$ to $\left(0,\left(M + 1\right)^{\frac{M + 1}{2}}/\left(2 (M - 1)^{\frac{M -1}{2}}\right) \right)$; see \cite{Forrester-Liu14}.

\begin{thm}[Bulk universality]\label{thm:bulk_truncated}
Let $K_n(x,y)$ be the correlation kernel defined in \eqref{Knintegralone}. For $x_0\in\left(0, \left(M + 1\right)^{\frac{M + 1}{2}}/ \left(2\left(M - 1\right)^{ \frac{M -1}{2}} \right) \right)$, which is parametrized by $\varphi \in (0, 2\pi/(M+1))$ through \eqref{eq:paraunitary}, we have, with $M \geq 2$ and $\nu_1,\ldots,\nu_M,\kappa$ being fixed,
\begin{align}\label{eq:bulk_unitary}
 \lim_{n \to \infty} \frac{e^{-\pi \xi\cot\frac{\varphi}{2}  }}{e^{-\pi \eta\cot\frac{\varphi}{2}  }} \frac{n^{M-2}}{\rho(\varphi)} K_n\left(n^{M-1}\left( x_0+ \frac{  \xi}{n\rho(\varphi) }\right), n^{M-1}\left( x_0+ \frac{  \eta}{n\rho(\varphi)}\right)\right)  =K_{\sin}(\xi,\eta)
\end{align}
uniformly for $\xi$ and $\eta$ in any compact subset of $\mathbb{R}$, where the function $\rho$ is given by
\begin{equation} \label{rhounitary}
\rho (\varphi) = \frac{1}{\pi x_0}\left(\frac{ \sin \left( \frac{M+1}{2} \varphi \right) }{ \sin \left( \frac{M-1}{2}\varphi \right)  }\right)^{1/2}\sin\frac{\varphi}{2}.
\end{equation}
\end{thm}

When the reference point $x_0$ is taken to be the right ending point, we have
\begin{thm}[Soft edge universality]\label{thmunitary:edge}
  With the correlation kernel $K_n(x,y)$ defined in \eqref{Knintegralone}, we have, with $\nu_1,\ldots,\nu_M,\kappa$ being fixed,
  \begin{multline}\label{eq:soft_univ_unitary}
    \lim_{n \to \infty}
    \exp\left\{ \left(\frac{n}{2}\right)^{\frac{1}{3}}\frac{(M+1)^{\frac{1}{2}}}{(M-1)^{\frac{1}{6}}}(\eta- \xi)     \right\}  \,
    n^{M-\frac{5}{3}} c_2 K_n\left(n^{M-1}\left( x_* + \frac{c_2 \xi}{n^{\frac{2}{3}}}\right), n^{M-1}\left( x_* + \frac{c_2 \eta}{n^{\frac{2}{3}}}\right)\right) \\
    = K_{\Ai}(\xi, \eta)
  \end{multline}
  uniformly for $\xi$ and $\eta$ in any compact subset of $\mathbb{R}$, where
  \begin{equation}  \label{eq:varrho}
    x_* = \frac{\left(M + 1\right)^{\frac{M + 1}{2}}}{2\left(M - 1\right)^{ \frac{M -1}{2}}}, \quad \text{and} \quad
    c_2 = \frac{\left(M+1\right)^{\frac{M+1}{2}}}{2^{\frac{4}{3}}(M-1)^{\frac{M}{2}-\frac{7}{6}}}.
\end{equation}
\end{thm}


\begin{proof}[Sketched proofs of Theorems \ref{thm:bulk_truncated} and \ref{thmunitary:edge}]
We scale the values of $x$ and $y$ in \eqref{Knintegralone} such that
\begin{equation}
  x=n^{M-1}\left( x_0+ \frac{  \xi}{\rho(\varphi) n}\right), \qquad y=n^{M-1}\left( x_0+ \frac{  \eta}{\rho(\varphi) n}\right)
\end{equation}
in the bulk case or
\begin{equation}
  x=n^{M-1}\left( x_* + \frac{c_2 \xi}{n^{\frac{2}{3}}} \right), \qquad y=n^{M-1}\left( x_* + \frac{c_2 \eta}{n^{\frac{2}{3}}} \right)
\end{equation}
in the soft edge case, where $\xi, \eta\in\mathbb{R}$. By using Stirling's formula for gamma functions and the reflection formula, in the bulk case it follows that, for $n$ large
\begin{align}
 K_n &\big(x,y\big) \sim \frac{n^{-M+1}}{(2\pi i)^2}\int_{\mathcal{C}} \ud s \oint_{\Sigma}  \ud t \frac{e^{n(\hat{F}(ns; x_0)-\hat{F}(nt; x_0))}}{s-t} \, \left(1+\frac{\xi}{n x_0 \rho}\right)^t \left(1+\frac{\eta}{n x_0 \rho }\right)^{-s} \nonumber\\
      &\left(x_0+\frac{\eta}{\rho n}\right)^{-1}\, \exp\left\{\sum_{j=0}^M\left(\nu_j+\frac{1}{2}\right)\log\frac{s}{t}-\left(\kappa-\frac{1}{2}\right) \log\frac{n+s}{n+t}-\frac{1}{2}\log\frac{s-n}{t-n}\right\},
\end{align}
where the contours $\mathcal{C}$ and $\Sigma$ depend on $x_0$ and
\begin{equation}
  \hat{F}(z; x_0)=(M+1)(z\log z-1)-(1+z)(\log(1+ z)-1)-(z-1)(\log(z-1)-1)-z\log x_0.
\label{Ffunctionunitary}
\end{equation}
Since
  \begin{equation}
  \hat{F}_z(z; x_0)=(M+1)\log z- \log(1+z)-\log(z-1)-\log x_0,
  \end{equation}
the saddle point of $F$ satisfies the algebraic  equation
\begin{equation}
z^{M+1}+x_0 (1-z^2)=0.
\label{aqunitary}
\end{equation}
One can find two explicit solutions of this equation with the help of \eqref{eq:paraunitary}, which are given by
\begin{equation} \label{twosunitary}
  w_{\pm}=\left(\frac{ \sin \left( \frac{M+1}{2} \varphi \right) }{ \sin\left( \frac{M-1}{2}\varphi \right)  }\right)^{1/2}e^{\pm i\frac{\varphi}{2}};
\end{equation}
cf. \cite{Forrester-Liu14}. The contours $\mathcal{C}$ and $\Sigma$ are chosen to be
\begin{equation}
  \mathcal{C} =\left\{z~\Big{|}~\Re \frac{z}{n} =\Re z_\pm=\left(\frac{ \sin \left( \frac{M+1}{2} \varphi \right) }{ \sin \left( \frac{M-1}{2}\varphi \right)  }\right)^{1/2}\cos\frac{\varphi}{2}\right\},\end{equation}
and the deformation of
\begin{equation}
\tilde{\Sigma}=\left\{
z=\left(\frac{ \sin ( \frac{M+1}{2} \phi) } { \sin( \frac{M-1}{2}\phi)  }\right)^{1/2} \,  e^{i\frac{\phi}{2}}
~\Big{|}~ -\frac{2\pi}{M+1} \leq \phi\leq \frac{2\pi}{M+1}
\right\},
\end{equation}
in manners similar to the construction of $\Sigma$ based on $\tilde{\Sigma}$ that we described in Sections  \ref{subsec:contours_bulk} and \ref{subsec:proof_of_bulk_univ}.

We then have $\Re \hat{F}(z; x_0)$ defined in \eqref{Ffunctionunitary} attains its global maximum at  $z=w_{\pm}$ for $nz\in \mathcal{C}$ and its global minimum at $z = w_{\pm}$ for $z \in \tilde \Sigma$. Thus, if $x_0\in\left(0, \left(M + 1\right)^{\frac{M + 1}{2}}/ \left( 2\left(M - 1\right)^{ \frac{M -1}{2}} \right) \right)$, like \eqref{eq:Kn_icorr_inverse},
 \begin{align}
  K_n \left(x,y \right) &\sim   \frac{n^{-M+1}}{ 2\pi i x_0}   \int_{nw_{-}}^{nw_+} \ud s  \, \left(1+\frac{\xi}{x_0\rho n}\right)^s \left(1+\frac{\eta}{x_0\rho n}\right)^{-s}\nonumber \\
  &\sim n^{-M+2}\, \rho\, e^{\pi (\xi-\eta)\cot\frac{\varphi}{2}}\, \frac{\sin\pi (\xi-\eta)}{\pi(\xi-\eta)},
 \end{align}
which is \eqref{eq:bulk_unitary}.

As $x_0 \to x_*$, we have  $z_{+}=z_-:=z_0=\sqrt{(M+1)/(M-1)}$. In this case, the integration over the contours around $z_0$ contributes the most.
Note that
\begin{equation}
  \hat{F}(z; x_*) = \hat{F}(z_0; x_*)+\frac{(M-1)^2}{6}(z-z_0)^3+\cdots, \qquad z\to z_0.
\end{equation}
With the formula
\begin{align}
 K_n &\big(x,y\big) \sim \frac{n^{-M+1}}{(2\pi i)^2}\int_{\mathcal{C}} \ud s \oint_{\Sigma}  \ud t \frac{e^{n(\hat{F}(ns; x_0)-\hat{F}(nt; x_0))}}{s-t} \, \left(1+\frac{c_2 \xi}{n^{\frac{2}{3}} x_*}\right)^t \left(1+\frac{c_2 \eta}{n^{\frac{2}{3}} x_*}\right)^{-s} \nonumber\\
      &\left(x_* + \frac{c_2 \eta}{n^{\frac{2}{3}}}\right)^{-1}\, \exp\left\{\sum_{j=0}^M\left(\nu_j+\frac{1}{2}\right)\log\frac{s}{t}-\left(\kappa-\frac{1}{2}\right) \log\frac{n+s}{n+t}-\frac{1}{2}\log\frac{s-n}{t-n}\right\},
\end{align}
by the change of variables
\begin{equation}
  s = nz_0 + n^{\frac{2}{3}} ((M-1)^{2}/2)^{-1/3}u, \qquad t = nz_0 + n^{\frac{2}{3}} ((M-1)^{2}/2)^{-1/3}v,
\end{equation}
and by the deformation of $\mathcal{C}$ and $\Sigma$ such that they go through the vicinity of $nz_0$ in proper directions, we have
\begin{align}
  K_n \left(x,y \right)   &\sim n^{-M+\frac{5}{3}}\, c^{-1}_2 \, \left(1+\frac{c_2 \xi}{x_* n^{\frac{2}{3}}} \right)^{nz_0} \left(1+\frac{c_2 \eta}{x_* n^{\frac{2}{3}}}\right)^{-nz_0} \,K_{\Ai}(\xi, \eta).
\end{align}
Thus \eqref{eq:soft_univ_unitary} is proved.
\end{proof}
%
%


\begin{rmk}
  By setting  $\xi=\eta=0$ in \eqref{eq:bulk_inverse} and \eqref{eq:bulk_unitary}, the bulk limit also implies point-wise convergence of one-point correlation functions in the support of the limiting measure. The functions $\rho(\varphi)$'s in \eqref{eq:bulk_inverse} and \eqref{eq:bulk_unitary} are actually density functions of the limiting spectral distribution for the squared singular values under proper parametrizations. Similar result holds for products of Ginibre matrices; see Theorem \ref{thm:bulk}. Thus we recover the limiting mean density results of the random matrix models discussed above, which were previously derived by the moment method of Stieltjes transforms; see e.g.~\cite{Alexeev-Gotze-Tikhomirov10}, \cite{Forrester14}. However, the moment method has the advantage in the discovery of natural parametrizations like \eqref{eq:para x}, \eqref{eq:parainverse} and \eqref{eq:paraunitary} by combinatorial relations; see \cite{Biane98}, \cite{Forrester-Liu14}, \cite{Haagerup-Moller13}, \cite{Neuschel14}.
\end{rmk}

\section*{Acknowledgment}
We thank the anonymous referees for their careful reading and constructive suggestions. The work of D.-Z. Liu was supported by the National Natural Science Foundation of China (Grants 11301499 and 11171005), and by the Fundamental Research Funds for the Central Universities (Grant WK0010000048). The work of D. Wang was partially supported by the start-up grant R-146-000-164-133. The work of L. Zhang was partially supported by The Program for Professor of Special Appointment (Eastern Scholar) at Shanghai Institutions of Higher Learning (No. SHH1411007) and by Grant SGST 12DZ 2272800, EZH1411513 from Fudan University.



\begin{thebibliography}{10}

\bibitem{Adhikari-Reddy-Reddy-Saha13}
K.~Adhikari, N.~K. Reddy, T.~R. Reddy, and K.~Saha.
\newblock Determinantal point processes in the plane from products of random
  matrices, 2013.
\newblock arXiv:1308.6817, to appear in Ann. Inst. Henri Poincar\'e Probab.
  Stat.

\bibitem{Adler-van_Moerbeke-Wang11}
M.~Adler, P.~van Moerbeke, and D.~Wang.
\newblock Random matrix minor processes related to percolation theory.
\newblock {\em Random Matrices Theory Appl.}, 2(4):1350008, 72, 2013.

\bibitem{Akemann-Burda12}
G.~Akemann and Z.~Burda.
\newblock Universal microscopic correlation functions for products of
  independent {G}inibre matrices.
\newblock {\em J. Phys. A}, 45(46):465201, 18, 2012.

\bibitem{Akemann-Burda-Kieburg-Nagao14}
G.~Akemann, Z.~Burda, M.~Kieburg, and T.~Nagao.
\newblock Universal microscopic correlation functions for products of truncated
  unitary matrices.
\newblock {\em J. Phys. A}, 47(25):255202, 26, 2014.

\bibitem{Akemann-Ipsen15}
G.~Akemann and J.~R. Ipsen.
\newblock Recent exact and asymptotic results for products of independent
  random matrices, 2015.
\newblock arXiv:1502.01667.

\bibitem{Akemann-Ipsen-Kieburg13}
G.~Akemann, J.~R. Ipsen, and M.~Kieburg.
\newblock Products of rectangular random matrices: {S}ingular values and
  progressive scattering.
\newblock {\em Phys. Rev. E}, 88(5):052118, 13, 2013.

\bibitem{Akemann-Kieburg-Wei13}
G.~Akemann, M.~Kieburg, and L.~Wei.
\newblock Singular value correlation functions for products of {W}ishart random
  matrices.
\newblock {\em J. Phys. A}, 46(27):275205, 22, 2013.

\bibitem{Alexeev-Gotze-Tikhomirov10}
N.~Alexeev, F.~G{\"o}tze, and A.~Tikhomirov.
\newblock Asymptotic distribution of singular values of powers of random
  matrices.
\newblock {\em Lith. Math. J.}, 50(2):121--132, 2010.

\bibitem{Anderson-Guionnet-Zeitouni10}
G.~W. Anderson, A.~Guionnet, and O.~Zeitouni.
\newblock {\em An introduction to random matrices}, volume 118 of {\em
  Cambridge Studies in Advanced Mathematics}.
\newblock Cambridge University Press, Cambridge, 2010.

\bibitem{Banica-Belinschi-Capitaine-Collins11}
T.~Banica, S.~T. Belinschi, M.~Capitaine, and B.~Collins.
\newblock Free {B}essel laws.
\newblock {\em Canad. J. Math.}, 63(1):3--37, 2011.

\bibitem{Beals-Szmigielski13}
R.~Beals and J.~Szmigielski.
\newblock Meijer {$G$}-functions: a gentle introduction.
\newblock {\em Notices Amer. Math. Soc.}, 60(7):866--872, 2013.

\bibitem{Bertola-Bothner14}
M.~Bertola and T.~Bothner.
\newblock Universality conjecture and results for a model of several coupled
  positive-definite matrices.
\newblock {\em Comm. Math. Phys.}, 337(3):1077--1141, 2015.

\bibitem{Bertola-Gekhtman-Szmigielski14}
M.~Bertola, M.~Gekhtman, and J.~Szmigielski.
\newblock Cauchy-{L}aguerre two-matrix model and the {M}eijer-{G} random point
  field.
\newblock {\em Comm. Math. Phys.}, 326(1):111--144, 2014.

\bibitem{Biane98}
P.~Biane.
\newblock Processes with free increments.
\newblock {\em Math. Z.}, 227(1):143--174, 1998.

\bibitem{Borodin99}
A.~Borodin.
\newblock Biorthogonal ensembles.
\newblock {\em Nuclear Phys. B}, 536(3):704--732, 1999.

\bibitem{Bougerol-Lacroix85}
P.~Bougerol and J.~Lacroix.
\newblock {\em Products of random matrices with applications to {S}chr\"odinger
  operators}, volume~8 of {\em Progress in Probability and Statistics}.
\newblock Birkh\"auser Boston, Inc., Boston, MA, 1985.

\bibitem{Burda-Janik-Waclaw10}
Z.~Burda, R.~A. Janik, and B.~Waclaw.
\newblock Spectrum of the product of independent random {G}aussian matrices.
\newblock {\em Phys. Rev. E (3)}, 81(4):041132, 12, 2010.

\bibitem{Burda-Jarosz-Livan-Nowak-Swiech10}
Z.~Burda, A.~Jarosz, G.~Livan, M.~A. Nowak, and A.~Swiech.
\newblock Eigenvalues and singular values of products of rectangular {G}aussian
  random matrices.
\newblock {\em Phys. Rev. E (3)}, 82(6):061114, 10, 2010.

\bibitem{Claeys-Kuijlaars-Wang15}
T.~Claeys, A.~B.~J. Kuijlaars, and D.~Wang.
\newblock Correlation kernels for sums and products of random matrices, 2015. 
\newblock arXiv:1505.00610.

\bibitem{Crisanti-Paladin-Vulpiani93}
A.~Crisanti, G.~Paladin, and A.~Vulpiani.
\newblock {\em Products of random matrices in statistical physics}, volume 104
  of {\em Springer Series in Solid-State Sciences}.
\newblock Springer-Verlag, Berlin, 1993.
\newblock With a foreword by Giorgio Parisi.

\bibitem{Forrester10}
P.~J. Forrester.
\newblock {\em Log-gases and random matrices}, volume~34 of {\em London
  Mathematical Society Monographs Series}.
\newblock Princeton University Press, Princeton, NJ, 2010.

\bibitem{Forrester14}
P.~J. Forrester.
\newblock Eigenvalue statistics for product complex {W}ishart matrices.
\newblock {\em J. Phys. A}, 47(34):345202, 22, 2014.

\bibitem{Forrester-Liu14}
P.~J. Forrester and D.-Z. Liu.
\newblock Raney distributions and random matrix theory.
\newblock {\em J. Stat. Phys.}, 158(5):1051--1082, 2015.

\bibitem{Forrester-Liu15}
P.~J. Forrester and D.-Z. Liu.
\newblock Singular values for products of complex {G}inibre matrices with a
  source: hard edge limit and phase transition, 2015.
\newblock arXiv:1503.07955.

\bibitem{Forrester-Wang15}
P.~J. Forrester and D.~Wang.
\newblock Muttalib--{B}orodin ensembles in random matrix theory ---
  realisations and correlation functions, 2015.
\newblock arXiv:1502:07147.

\bibitem{Furstenberg-Kesten60}
H.~Furstenberg and H.~Kesten.
\newblock Products of random matrices.
\newblock {\em Ann. Math. Statist.}, 31:457--469, 1960.

\bibitem{Gotze-Kosters-Tikhomirov14}
F.~G{\"o}tze, H.~K{\"o}sters, and A.~Tikhomirov.
\newblock Asymptotic spectra of matrix-valued functions of independent random
  matrices and free probability, 2014.
\newblock arXiv:1408.1732, to appear in Random Matrices Theory Appl.

\bibitem{Gotze-Tikhomirov10}
F.~G{\"o}tze and A.~Tikhomirov.
\newblock On the asymptotic spectrum of products of independent random
  matrices, 2010.
\newblock arXiv:1012.2710.

\bibitem{Haagerup-Moller13}
U.~Haagerup and S.~M{\"o}ller.
\newblock The law of large numbers for the free multiplicative convolution.
\newblock In {\em Operator algebra and dynamics}, volume~58 of {\em Springer
  Proc. Math. Stat.}, pages 157--186. Springer, Heidelberg, 2013.

\bibitem{Ipsen-Kieburg14}
J.~R. Ipsen and M.~Kieburg.
\newblock Weak commutation relations and eigenvalue statistics for products of
  rectangular random matrices.
\newblock {\em Phys. Rev. E}, 89(3):032106, 20, 2014.

\bibitem{Ismail09}
M.~E.~H. Ismail.
\newblock {\em Classical and quantum orthogonal polynomials in one variable},
  volume~98 of {\em Encyclopedia of Mathematics and its Applications}.
\newblock Cambridge University Press, Cambridge, 2009.
\newblock With two chapters by Walter Van Assche, With a foreword by Richard A.
  Askey, Reprint of the 2005 original.

\bibitem{Jiang09}
T.~Jiang.
\newblock Approximation of {H}aar distributed matrices and limiting
  distributions of eigenvalues of {J}acobi ensembles.
\newblock {\em Probab. Theory Related Fields}, 144(1-2):221--246, 2009.

\bibitem{Johansson06}
K.~Johansson.
\newblock Random matrices and determinantal processes.
\newblock In {\em Mathematical statistical physics}, pages 1--55. Elsevier B.
  V., Amsterdam, 2006.

\bibitem{Kanwal97}
R.~P. Kanwal.
\newblock {\em Linear integral equations}.
\newblock Birkh\"auser Boston, Inc., Boston, MA, second edition, 1997.

\bibitem{Krishnapur06}
M.~R. Krishnapur.
\newblock {\em Zeros of random analytic functions}.
\newblock ProQuest LLC, Ann Arbor, MI, 2006.
\newblock Thesis (Ph.D.)--University of California, Berkeley.

\bibitem{Kuijlaars11}
A.~B.~J. Kuijlaars.
\newblock Universality.
\newblock In {\em The {O}xford handbook of random matrix theory}, pages
  103--134. Oxford Univ. Press, Oxford, 2011.

\bibitem{Kuijlaars-Stivigny14}
A.~B.~J. Kuijlaars and D.~Stivigny.
\newblock Singular values of products of random matrices and polynomial
  ensembles.
\newblock {\em Random Matrices Theory Appl.}, 3(3):1450011, 22, 2014.

\bibitem{Kuijlaars-Zhang14}
A.~B.~J. Kuijlaars and L.~Zhang.
\newblock Singular values of products of {G}inibre random matrices, multiple
  orthogonal polynomials and hard edge scaling limits.
\newblock {\em Comm. Math. Phys.}, 332(2):759--781, 2014.

\bibitem{Liu-Song-Wang11}
D.-Z. Liu, C.~Song, and Z.-D. Wang.
\newblock On explicit probability densities associated with {F}uss-{C}atalan
  numbers.
\newblock {\em Proc. Amer. Math. Soc.}, 139(10):3735--3738, 2011.

\bibitem{Luke69}
Y.~L. Luke.
\newblock {\em The special functions and their approximations, {V}ol. {I}}.
\newblock Mathematics in Science and Engineering, Vol. 53. Academic Press, New
  York-London, 1969.

\bibitem{Neuschel14}
T.~Neuschel.
\newblock Plancherel-{R}otach formulae for average characteristic polynomials
  of products of {G}inibre random matrices and the {F}uss-{C}atalan
  distribution.
\newblock {\em Random Matrices Theory Appl.}, 3(1):1450003, 18, 2014.

\bibitem{Nica-Speicher06}
A.~Nica and R.~Speicher.
\newblock {\em Lectures on the combinatorics of free probability}, volume 335
  of {\em London Mathematical Society Lecture Note Series}.
\newblock Cambridge University Press, Cambridge, 2006.

\bibitem{Boisvert-Clark-Lozier-Olver10}
F.~W.~J. Olver, D.~W. Lozier, R.~F. Boisvert, and C.~W. Clark, editors.
\newblock {\em N{IST} handbook of mathematical functions}.
\newblock U.S. Department of Commerce, National Institute of Standards and
  Technology, Washington, DC; Cambridge University Press, Cambridge, 2010.
\newblock With 1 CD-ROM (Windows, Macintosh and UNIX).

\bibitem{ORourke-Soshnikov11}
S.~O'Rourke and A.~Soshnikov.
\newblock Products of independent non-{H}ermitian random matrices.
\newblock {\em Electron. J. Probab.}, 16:no. 81, 2219--2245, 2011.

\bibitem{Penson-Zyczkowski11}
K.~A. Penson and K.~{\.Z}yczkowski.
\newblock Product of {G}inibre matrices: {F}uss-{C}atalan and {R}aney
  distributions.
\newblock {\em Phys. Rev. E}, 83(6):061118, 9, 2011.

\bibitem{Soshnikov00}
A.~Soshnikov.
\newblock Determinantal random point fields.
\newblock {\em Uspekhi Mat. Nauk}, 55(5(335)):107--160, 2000.

\bibitem{Tracy-Widom94b}
C.~A. Tracy and H.~Widom.
\newblock Level spacing distributions and the {B}essel kernel.
\newblock {\em Comm. Math. Phys.}, 161(2):289--309, 1994.

\bibitem{Tulino-Verdu04}
A.~M. Tulino and S.~Verd{\'u}.
\newblock Random matrix theory and wireless communications.
\newblock {\em Found. Trends Commun. Inform. Theory}, 1(1):1--182, 2004.

\bibitem{Zhang13}
L.~Zhang.
\newblock A note on the limiting mean distribution of singular values for
  products of two {W}ishart random matrices.
\newblock {\em J. Math. Phys.}, 54(8):083303, 8, 2013.

\bibitem{Zhang15}
L.~Zhang.
\newblock Local universality in biorthogonal {L}aguerre ensembles, 2015.
\newblock arXiv:1502.03160.

\end{thebibliography}

\end{document}